\documentclass[11pt]{amsart}
\usepackage{amsfonts,amssymb,amsthm}
\usepackage[numbers,sort&compress]{natbib}
\usepackage[leqno]{amsmath}
\usepackage[mathscr]{euscript}
\usepackage{tikz}
\usetikzlibrary{arrows.meta, positioning} % 可选：加载箭头和定位库
\usepackage{graphicx} % Allows including images
\usepackage{booktabs}
\usepackage{color}
\usepackage{mathrsfs}
\usepackage{paralist}
\usepackage{hyperref}
\usepackage{framed}
\usepackage{epstopdf}
\usepackage{algorithm}
\usepackage{algorithmic}
\usepackage{caption}
\newtheorem {theorem}{Theorem}[section]
\newtheorem {proposition}{Proposition}[section]
\newtheorem {corollary}{Corollary}[section]

\newtheorem {lemma}{Lemma}[section]
\newtheorem {example}{Example}[section]
\newtheorem {definition}{Definition}[section]
\newtheorem {remark}{Remark}[section]

\newcommand{\R}{{\mathbb R}}
\newcommand{\N}{\mathbb{N}}
\newcommand{\M}{\mathbf{M}}

\newcommand{\bx}{\textrm{\bf{x}}}

\newcommand{\bz}{\textrm{\bf{z}}}
\newcommand{\bb}{\textrm{\bf{b}}}
\newcommand{\by}{\textrm{\bf{y}}}
\newcommand{\bw}{\textrm{\bf{w}}}
\newcommand{\bze}{\boldsymbol 0}
\newcommand{\blambda}{{\boldsymbol \lambda}}
\newcommand{\ba}{{\boldsymbol \alpha}}
\newcommand{\bbe}{{\boldsymbol \beta}}
\newcommand{\bxi}{{\boldsymbol \xi}}
\newcommand{\bzeta}{{\boldsymbol \zeta}}

\makeatletter
\newcommand{\leqnomode}{\tagsleft@true}
\newcommand{\reqnomode}{\tagsleft@false}
\makeatother

\title{First-order SDSOS-convex semi-algebraic optimization and exact SOCP relaxations}

\author{Chengmiao Yang}
\address[Chengmiao Yang]{Academy for Advanced Interdisciplinary Studies, Northeast Normal University, Changchun 130024, Jilin Province, China}
\email{cmyang@nenu.edu.cn}

\author{Liguo Jiao}
\address[Liguo Jiao]{Academy for Advanced Interdisciplinary Studies, Northeast Normal University, Changchun 130024, Jilin Province, China; Shanghai Zhangjiang Institute of Mathematics, Shanghai 201203, China}
\email{jiaolg356@nenu.edu.cn; hanchezi@163.com}

\author{Jae Hyoung Lee$^{\dag}$}
\address[Jae Hyoung Lee]{Department of Applied Mathematics, Pukyong National University, Busan, 48513, Korea}
\email{mc7558@naver.com}

\thanks{$^{\dag}$Corresponding Author}

\date{\today}

\begin{document}
	
\begin{abstract}
In this paper, we define a new type of nonsmooth convex function, called {\em first-order SDSOS-convex semi-algebraic function}, which is an extension of the previously proposed first-order SDSOS-convex polynomials (Chuong et al. in J Global Optim 75:885--919, 2019).
This class of nonsmooth convex functions contains many well-known functions, such as the Euclidean norm,  the $\ell_1$-norm commonly used in compressed sensing and sparse optimization, and the least squares function frequently employed in machine learning and regression analysis.
We show that, under suitable assumptions, the optimal value and optimal solutions of first-order SDSOS-convex semi-algebraic programs can be found by exactly solving an associated second-order cone programming problem.
Finally, an application to robust optimization with first-order SDSOS-convex polynomials is discussed.

\end{abstract}

\subjclass[2020]{90C32; 90C22; 90C23}

\keywords{Convex optimization programs; second-order cone programming; first-order SDSOS-convex semi-algebraic function.}

\maketitle
	
\section{Introduction}\label{sect:1}
Consider a standard convex optimization problem of the following form
\begin{align*}
\min _{\bx \in \mathbb{R}^{n}} \left\{f_{0}(\bx) \colon  f_{i}(\bx)\le 0, \, i = 1, \ldots, m\right\},
\end{align*}
where $f_{i}\colon\mathbb{R}^n\to\mathbb{R},$  $i = 0, 1, \ldots, m,$ are convex functions.
An important property of convex functions is that a local optimum (if exists) is also a global one.
This characteristic has led to widespread applications of convex optimization in fields such as engineering, economics, and data science \cite{Boyd2004,Ben2001}.
However, determining whether a given function is convex---especially in high-dimensional settings---has been shown to be an NP-hard problem, even in the case of polynomials~\cite{Ahmadi2013-NP}.

To address this challenge, researchers have focused on polynomials with structured representations, and proposed the concept of SOS-convex polynomial \cite{Helton2010} as a computable sufficient condition for convexity.
A key feature of SOS-convex polynomials is that verifying SOS-convexity can be equivalently formulated as a semidefinite programming (SDP) feasibility problem, which can be efficiently solved in polynomial time~\cite{Lasserre2014}.
%However, the dependence of this technique on large-scale SDP limits the scale of problems to which it can be applied.
However, this technique relies on large-scale SDP, which limits the size of problems it can solve.
In order to reduce the computational complexity, Ahmadi and Majumdar~\cite{Ahmadi2019} recently introduced the concept of scaled diagonally dominant sum of squares (SDSOS) polynomials.
By restricting the Gram matrix in the SOS decomposition to be scaled diagonally dominant, SDSOS reformulates the SDP into a second-order cone programming (SOCP) problem~\cite{Kuang2019}, thereby improving the efficiency of solving large-scale problems.
The key efficiency gain arises from replacing the high-dimensional positive semidefinite constraint in SDP by a set of second-order cone constraints in SDSOS.
This substitution takes advantage of the fact that SOCP has similar computational complexity to linear programming~\cite{Ahmadi2019}, making it much faster to solve large-scale problems.

Very recently, building on this idea, a new notion of {\em first-order SDSOS-convex polynomials} was introduced by Chuong et al.~\cite{Chuong2019}, along with a tractable sufficient condition for convexity formulated via SOCP problems.
Remarkably, under certain assumptions, it has been shown that the first-order SDSOS-convex polynomial optimization problem admits finite convergence at the first level of the SOCP hierarchy, from which a global solution can be directly extracted.

{\small
\begin{figure}[ht]
	\centering
	\begin{tikzpicture}[>=latex, node distance=2.3cm, on grid]
		% 加载必要的库
		\usetikzlibrary{fit}
		
		% 定义节点样式 - 强制固定尺寸
		\tikzstyle{concept} = [rectangle, draw, rounded corners,
		minimum width=6cm,
		minimum height=1.2cm, % 增加高度以容纳两行文本
		text width=6cm, % 限制文本宽度防止溢出
		align=center]
		\tikzstyle{highlight} = [concept, fill=blue!10]
		
		% 第一条线（SOS）
		\node[concept] (sos) {SOS polynomial};
		\node[concept, below=of sos] (sos-conv) {SOS-convex polynomial \cite{Helton2010}};
		\node[concept, below=of sos-conv] (sos-salg) {SOS-convex semi-algebraic \\ function \cite{Chieu2018}};
		
		% 第二条线（SDSOS）
		\node[concept, right=7cm of sos] (sdsos) {SDSOS polynomial \cite{Ahmadi2019}};
		\node[concept, below=of sdsos] (sdsos-conv) {\parbox{5.5cm}{\centering First-order SDSOS-convex \\ polynomial \cite{Chuong2019}}}; % 使用parbox确保两行文本居中
		\node[highlight, below=of sdsos-conv] (sdsos-salg) {First-order SDSOS-convex \\ semi-algebraic function};
		
		% 垂直线连接
		\draw[->] (sos) -- (sos-conv);
		\draw[->] (sos-conv) -- (sos-salg);
		\draw[->] (sdsos) -- (sdsos-conv);
		\draw[->] (sdsos-conv) -- (sdsos-salg);
		
		% 水平对应关系（双向箭头虚线）
		\draw[<->, dashed] (sos) -- (sdsos);
		\draw[<->, dashed] (sos-conv) -- (sdsos-conv);
		\draw[<->, dashed] (sos-salg) -- (sdsos-salg);
		
		% 添加光滑区域框和标签
		\node[draw, dashed, rounded corners, inner sep=10pt, fit=(sos) (sdsos) (sos-conv) (sdsos-conv), label={[xshift=-0.5cm]left:\textrm{smooth}}] (smooth-region) {};
		
		% 添加非光滑区域框和标签
		\node[draw, dashed, rounded corners, inner sep=10pt, fit=(sos-salg) (sdsos-salg), label={[xshift=-0.5cm]left:\textrm{nonsmooth}}] (nonsmooth-region) {};
		\end{tikzpicture}
		\caption{Motivation of this work.}
		\label{fig:concept-lines}
\end{figure}
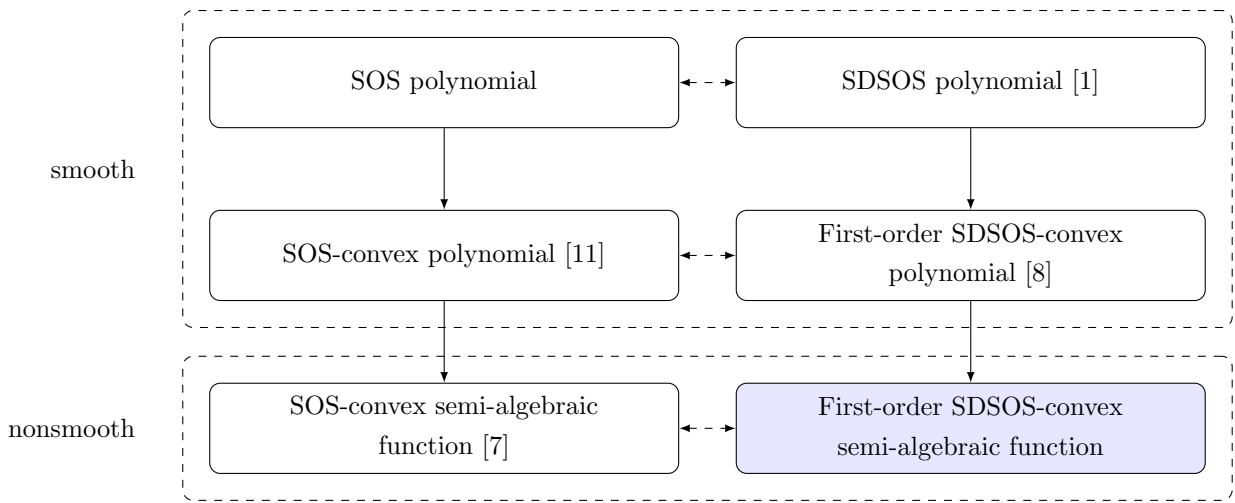}
\subsection{Motivations}
Although SOS-convex (resp., first-order SDSOS-convex) polynomials provide a verifiable framework for convex optimization via SDP (resp., SOCP) problems, their applicability remains limited to polynomial functions.
However, in practical optimization problems, non-polynomial and even nonsmooth structures are commonly encountered.
To overcome this limitation, researchers turned to the concept of semi-algebraic functions\footnote{A function $f: \mathbb{R}^{n} \rightarrow \mathbb{R}$ is said to be semi-algebraic if its graph $\left\{(\bx, y) \in \mathbb{R}^{n + 1} \colon y = f(\bx)\right\}$ forms a semi-algebraic set.}; see, e.g.,~\cite{HaHV2017,Lasserre2014}.
%Semi-algebraic functions can describe a broader class of functions (e.g., Euclidean norm).
%It is worth mentioning that
Motivated by which, a new class of nonsmooth convex functions, called {\em SOS-convex semi-algebraic functions}, has been proposed by Chieu et al.~\cite{Chieu2018}; moreover, this class of functions generalize the concept of SOS-convex polynomials.
It is also shown that, under commonly used strict feasibility conditions, optimization problems involving these functions can attain their optimal value and solutions by solving a single SDP program.
	
In this paper, motivated by \cite{Chieu2018,Chuong2019}, we aim to introduce and study a new class of {\em nonsmooth} convex functions, and we call it {\em first-order SDSOS-convex semi-algebraic function}; see {\sc Figure}~\ref{fig:concept-lines} for the motivation, and Section~\ref{sect:3} for its formal definition.

\subsection{Contributions}
Besides the introduced first-order SDSOS-convex semi-algebraic function, our major contributions are as follows:
\begin{itemize}
\item Exact SOCP relaxations for optimization problems with first-order SDSOS-convex semi-algebraic functions are studied. Namely, under suitable assumptions, the optimal value and optimal solutions can be obtained by solving an associated SOCP problem.
\item As applications, we show that, under some conditions, robust first-order SDSOS-convex polynomial optimization problems with SDD-restricted spectrahedron uncertainty set aslo enjoy exact SOCP relaxations.
\end{itemize}

The outline of the paper is organized as follows.
Section~\ref{sect:2} presents some notation and preliminaries.
Main results with some applications are proposed in sections~\ref{sect:3}, \ref{sect:4} and \ref{sect:5}.
Conclusions are given in Sect.~\ref{sect:6}.

\section{Notation and preliminaries}\label{sect:2}
Let us recall some notation and preliminary results that will be frequently used throughout this paper.
Let $\mathbb{R}^{n}$ denote the $n$-dimensional Euclidean space with the inner product $\langle\cdot, \cdot\rangle$ and the associated Euclidean norm $\|\cdot\|.$
We suppose $1 \leq n \in \mathbb{N},$ where $\mathbb{N}$ stands for the set of nonnegative integers.
The non-negative orthant of $\mathbb{R}^{n}$ is denoted by $\mathbb{R}_{+}^{n}:=\left\{\bx \in \mathbb{R}^n \colon x_i \geq 0,\ i = 1, \ldots, n\right\}.$
	
%\subsection{Convex analysis}
For an extended real-valued function $f$ on $\mathbb{R}^{n},$ $f$ is said to be {\em proper} if for all $\bx\in \mathbb{R}^{n},$ $f(\bx) > -\infty$ and there exists $\widehat{\bx} \in \mathbb{R}^{n}$ such that $f(\widehat{\bx}) \in \mathbb{R}.$
We denote its domain and epigraph of $f$ by ${\rm dom\,} f:=\{\bx\in \mathbb{R}^{n}\colon f(\bx) < +\infty\}$ and ${\rm epi\,}f:=\{(\bx,r)\in \mathbb{R}^{n}\times\mathbb{R} \colon f(\bx) \leq r\},$ respectively.
%	We say a function $f$ is  lower semicontinuous (l.s.c.) if $\liminf_{\by\to \bx}f(\by)\geq f(\bx)$ for all $\bx\in\mathbb{R}^{n} ;$ in other words, ${\rm epi\,}f$ is closed in $\mathbb{R}^{n}$; see, e.g., \cite[Theorem 7.1]{Rockafellar1970}.
A function $f\colon\mathbb{R}^{n} \to \mathbb{R} \cup \{+\infty\}$ is said to be {\em convex} if
%for all $t \in [0, 1],$ $f((1-t)\bx + t \by)\leq (1-t) f(\bx) + t f(\by)$ for all $\bx, \by \in \mathbb{R}^{n};$ or equivalently
${\rm epi\,}f$ is a convex set in $\mathbb{R}^n \times \mathbb{R}.$
As usual, for any proper convex function $f$ on $\mathbb{R}^{n},$ its conjugate function $f^*: \mathbb{R}^{n} \to \mathbb{R}\cup \{+\infty\}$ is defined by $f^*(\bx^*) = \sup \left\{\langle \bx^*,\bx\rangle-f(\bx)\colon \bx\in\mathbb{R}^{n}\right\}$ for all $\bx^*\in\mathbb{R}^{n}.$

For a given set $A \subset \mathbb{R}^{n},$ we denote the closure and the convex hull generated by $A$ by ${\rm cl\,}A$ and ${\rm co\,}A,$ respectively.
The indicator function $\delta_A: \mathbb{R}^{n} \rightarrow \mathbb{R} \cup\{+\infty\}$ is defined by
	\begin{align*}
		\delta_A(\bx):=\left\{
		\begin{array}{@{\,}ll}
			0, & \bx\in A,\\
			+\infty, & {\rm otherwise}.
		\end{array}
		\right.
	\end{align*}
Note that, if $A$ is convex, then so is the indicator function $\delta_A.$

In what follows, we recall some results from real polynomials and SDP problems.
The space of all real polynomials in the variable $\bx$ is denoted by $\mathbb{R}[\bx];$
moreover, the space of all real polynomials in the variable $\bx$ with degree at most $d$ is denoted by $\mathbb{R}[\bx]_{d},$ where the degree of a polynomial $f$ is denoted by $\deg f.$
We say that a real polynomial $f$ is {\em sum of squares} (SOS in short), if there exist real polynomials $q_{\ell},$ $\ell = 1, \ldots, s$ such that $f =\sum_{\ell=1}^{s}q_{\ell}^{2}.$
The set of all sums of squares of real polynomials of degree at most $2d$ in the variable $\bx$ is denoted by $\Sigma[\bx]_{2d}.$
For a multi-index $\ba\in\mathbb{N}^{n},$ let $|\ba|:=\sum_{i=1}^n\alpha_{i},$ and let $\mathbb{N}^n_{d}:= \{\ba \in \mathbb{N}^n \colon  |\ba|\leq d\},$
we denote by $\bx^\ba$ the monomial $x_{1}^{\alpha_{1}}\cdots x_{n}^{\alpha_n}.$
The canonical basis of $\mathbb{R}[\bx]_{d}$ is denoted by
\begin{align}\label{cano_basis}
\lceil \bx\rceil_d:=(\bx^\ba)_{\ba\in\mathbb{N}^n_{d}}=(1,x_{1},\ldots,x_{n},x_{1}^{2},x_{1}x_{2},\ldots,x_{n}^{2},\ldots,x_{1}^{d},\ldots,x_{n}^{d})^{T},
\end{align}
which has dimension $s(n,d):=\left( \substack{ n+d \\ n }\right).$
Given an $s(n, 2d)$-vector $\by:= (y_\ba)_{\ba\in\N^n_{2d}}$ with $y_{\bze}=1,$ let $\M_{d}(\by)$ be the moment matrix of dimension $s(n, d),$ with rows and columns labeled by \eqref{cano_basis} in the sense that
\begin{align*}
\M_d(\by)(\ba,\bbe)=y_{\ba+\bbe}, \ \forall \ba,\bbe\in\N^n_d.
\end{align*}
In other words, $\M_d(\by)=\sum_{\ba\in\N^n_{2d}}y_\ba (\lceil \bx\rceil_d\lceil \bx\rceil_d^T)_\ba.$ The gradient and the Hessian of a polynomial $f\in \mathbb{R}[\bx]$ at a point $\overline \bx$ are denoted by $\nabla f(\overline \bx)$ and $\nabla^{2}f(\overline \bx),$ respectively.
	
Let $S^{n}$ be the set of $n\times n$ symmetric matrices.
The notion ``$\succeq$" denotes the L\"{o}wner partial order of $S^{n},$ that is, for $X,Y\in S^{n},$ $X\succeq Y$ if and only if $X-Y$ is a positive semidefinite matrix.
Let $S_{+}^{n}$ be the set of $n\times n$ symmetric positive semidefinite matrices.
Also, $X \succ 0$ means that $X$ is a positive definite matrix.
For $X, Y\in S^{n},$ $\langle X, Y\rangle := {\rm tr}(XY),$ where ``${\rm tr}$'' denotes the trace of a matrix.
	
Recall that a symmetric matrix $A=\left(a_{i j}\right)$ is called diagonally dominant ($dd$ in short) if
	\begin{align*}
		a_{i i} \geq \sum_{j \neq i}\left|a_{i j}\right| \ \text { for all } \ i.
	\end{align*}
If in addition the inequality is strict for every  $i ,$ i.e.,
	\begin{align*}
		a_{i i}>\sum_{j \neq i}\left|a_{i j}\right| \ \text { for all } \ i,
	\end{align*}
then $A$ is said to be strictly diagonally dominant.
A symmetric matrix $A$ is called scaled diagonally dominant ($sdd$ in short) if there exists a diagonal matrix $D,$ with positive diagonal entries, such that $DAD$ is diagonally dominant.
The set of $n\times n$ $dd$ and $sdd$ matrices are denoted by $S_{dd}^n$ and $S_{sdd}^n,$ respectively.
	
\medskip
%Below, let us recall a very interesting subclass of convex polynomials in $\mathbb{R}[\bx]$ introduced by Helton and Nie~\cite{Helton2010} (see also \cite{Ahmadi2012,Ahmadi2013}).
The following concept was introduced by Helton and Nie~\cite{Helton2010} (see also \cite{Ahmadi2012,Ahmadi2013}).
\begin{definition}[SOS-convex polynomial] %\label{SOS-convexity}
{\rm A real polynomial $f$ on $\mathbb{R}^{n}$ is called {\it SOS-convex} if the polynomial $F: \ (\bx, \by) \mapsto f(\bx) - f(\by) - \nabla f(\by)^{T}(\bx - \by)$ is an SOS polynomial on $\mathbb{R}^{n} \times \mathbb{R}^{n}.$
}\end{definition}
Recently, Chieu et al.~\cite{Chieu2018} introduced a new concept called SOS-convex semi-algebraic function, which is a class of {\em nonsmooth} convex functions, and it includes SOS-convex polynomial as a special case.
	
\begin{definition}[SOS-convex semi-algebraic function]\label{sos-convex-sa}
{\rm A real function $f: \mathbb{R}^{n} \rightarrow \mathbb{R}$ is said to be an {\it SOS-convex semi-algebraic function }if it admits a representation
\begin{align*}
f(\bx) = \max_{\by \in \Omega}\left\{h_{0} (\bx) + \sum_{j=1}^{s} y_{j} h_{j} (\bx) \right\},
\end{align*}
where
\begin{itemize}
\item[{\rm (i)}] each $h_{j},$ $j = 0,1, \ldots, s,$ is a polynomial and for each $\by \in \Omega,$  $h_{0} + \sum_{j=1}^{s} y_{j} h_{j}$ is an SOS-convex polynomial on $\mathbb{R}^{n};$
\item[{\rm (ii)}] $\Omega$ is a nonempty compact semidefinite program representable set given by
\begin{align*}
\Omega := \left\{\by \in \mathbb{R}^{s} \colon \exists \bz \in \mathbb{R}^{p} \text { such that } A_{0}+\sum_{j=1}^{s} y_{j} A_{j}+\sum_{l=1}^{p} z_{\ell} B_{\ell} \in S_+^q\right\},
\end{align*}
for some $p \in \mathbb{N}, A_{j},$ $j = 0,1, \ldots, s,$ and $B_{\ell},\ \ell=1, \ldots, p,$ being $(q\times q)$-symmetric matrices with some $q \in \mathbb{N}.$
\end{itemize}}
\end{definition}
	
To reduce the computational burden of SOS-based relaxations, Ahmadi and Majumdar~\cite{Ahmadi2019} proposed a novel concept called SDSOS polynomial, which can be viewed as a more tractable SOCP alternative.
\begin{definition}[SDSOS polynomial]
{\rm A real polynomial $f$ with degree $2d$ on $\mathbb{R}^n$ is called a {\em scaled diagonally dominant sum of squares {\rm (SDSOS in short)}} if it can be written as
\begin{align*}
f(\bx) = \sum_{i=1}^{p} \alpha_im^2_i(\bx) + \sum_{i,j=1}^p(\beta^+_{ij}m_i(\bx) + \gamma^+_{ij}m_j(\bx))^2 +\sum_{i,j=1}^p(\beta^-_{ij}m_i(\bx)-\gamma^-_{ij}m_j(\bx))^2
\end{align*}
for some monomials $m_{i}(\bx),$ $m_{j}(\bx)$ and some scalars $\alpha_i,\beta^+_{ij},\gamma^+_{ij},\beta^-_{ij},\gamma^-_{ij}$ with $\alpha_i \geq 0.$
We denote the set of SDSOS polynomials in $\bx$ with degree $2d$ by $\widetilde\Sigma[\bx]_{2d}.$
}\end{definition}

Observe that any SDSOS polynomial is SOS, but the converse may not be true in general; see, e.g.,~\cite{Ahmadi2019}.
An important feature of SDSOS polynomials is that the SDSOS property of a given real polynomial can be verified by solving a single SOCP problem~\cite{Ahmadi2019}.
The following proposition provides a reformulation of SDSOS polynomials as a conic programming problem, in particular, a second-order cone program.
\begin{proposition}\label{proposition1}{\rm \cite{Ahmadi2019}}
A polynomial $f\in\mathbb{R}[\bx]_{2d}$ has a scaled diagonally dominant sums-of-squares decomposition if and only if there exists $Q\in S^{s(n,d)}_{sdd}$ such that $f(\bx)=\langle \lceil \bx\rceil_d\lceil \bx\rceil_d^{T},Q\rangle$ for all $\bx\in\mathbb{R}^{n}.$
\end{proposition}

\begin{remark}{\rm
It is worth mentioning that, according to~\cite[Lemma~3.8 and Theorem~3.9]{Ahmadi2019}, a matrix $Q\in S^{{s(n,d)}}_{sdd}$ can be decomposed as
\begin{align*}
Q = \sum_{1\leq i<j\leq s(n,d)}M^{ij},
\end{align*}
where each matrix $M^{ij}$ has all entries zero except for the four positions $(i,i), (i,j), (j,i), (j,j),$ which form a symmetric $2 \times 2$ matrix:
\begin{align*}
\left(
		\begin{array}{cc}
			M^{ij}_{ii} & M^{ij}_{ij} \\
			M^{ij}_{ji} & M^{ij}_{jj} \\
		\end{array}
\right)\succeq0,
\end{align*}
which is equivalent to the second-order cone constraint:
\begin{align*}
		\left\|\left(
		\begin{array}{cc}
			2M^{ij}_{ij} \\
			M^{ij}_{ii} - M^{ij}_{jj} \\
		\end{array}
		\right)\right\|\leq M^{ij}_{ii}+M^{ij}_{jj}.
\end{align*}
}\end{remark}

Later on, motivated by~\cite{Ahmadi2019}, Chuong et al.~\cite{Chuong2019} introduced the following class of polynomials, which form a subclass of SOS-convex polynomials.
\begin{definition}[first-order SDSOS-convex polynomial]%\label{FOSDSOS}
{\rm A real polynomial $f$ on $\mathbb{R}^n$ is called {\em first-order SDSOS-convex} if $h_f$ is an SDSOS polynomial in the variables  $\bx$ and $\by,$ where $h_f$ is a polynomial defined by $h_f(\bx, \by) := f(\bx)-f(\by)- \langle \nabla f(\by), \bx-\by\rangle.$
}
\end{definition}
	
It has been shown that the first-order SDSOS-convexity is a sufficient condition for SOS-convex polynomials, however, the converse may not be true; see, e.g., \cite[Example~5.4]{Chuong2019}.
Moreover, checking whether a given polynomial $f$ is first-order SDSOS-convex or not can be done by solving a second-order cone programming problem.
These polynomials exhibit several notable properties; see, e.g., the next Lemmas~\ref{lemma4} and \ref{lemma5}.
%Remarkably, Chuong et al. \cite{Chuong2019} also proposed an important property of first-order SDSOS-convex polynomials, we recall it as below.
\begin{lemma}\label{lemma4} {\rm \cite[Proposition~5.3]{Chuong2019}}
Let $f$ be a first-order SDSOS-convex polynomial.
Then $f(\bx)\geq0$ for all $\bx\in\mathbb{R}^n$ if and only if $f$ is SDSOS.
\end{lemma}

%\textcolor[rgb]{1.00,0.00,0.00}{ Lemma 2.1 provides present an SDSOS representation result for nonnegative first-order SDSOS-convex polynomials.Lemma 2.2 establishes a Jensen-type inequality for such polynomials, which plays a key role in extracting solutions from SOCP relaxations.}

\begin{lemma}\label{lemma5} {\rm \cite[Proposition~6.1]{Chuong2019}}
Let $f$ be a first-order SDSOS-convex polynomial on $\mathbb{R}^n$ with degree $2d,$ and let $\by=(y_{\ba})_{\ba\in\mathbb{N}_{2d}^n}$ satisfy $y_{\bze}=1$ and
\begin{align*}
		\left\|\begin{array}{c}
			2(\M_{d}(\by))_{ij}\\
			(\M_{d}(\by))_{ii}-(\M_{d}(\by))_{jj} \\
		\end{array}\right\|\leq (\M_{d}(\by))_{ii}+(\M_{d}(\by))_{jj}, \ \ 1\leq i, j\leq s(n,d).
\end{align*}
Let $L_{\by}: \mathbb{R}[\bx]\to\mathbb{R}$ be a linear functional defined by $L_{\by}(f):=\sum\limits_{\ba\in\mathbb{N}_{2d}^n} f_{\ba} y_{\ba},$
where $f(\bx)=\sum\limits_{\ba\in\mathbb{N}_{2d}^n} f_{\ba} \bx^{\ba}.$
Then
\begin{align*}
L_{\by}(f)\geq f(L_{\by}(\bx)),
\end{align*}
where $L_{\by}(\bx):=(L_{\by}(x_{1}),\ldots,L_{\by}(x_{n})).$
\end{lemma}

\section{First-order SDSOS-convex semi-algebraic functions}\label{sect:3}
%Motivated by~\cite{Chieu2018, Chuong2019},
In this section, we introduce a new class of functions called first-order SDSOS-convex semi-algebraic functions (see the {\sc Figure}~\ref{fig:concept-lines} as shown in Section~\ref{sect:1} for the motivation).
The class of first-order SDSOS-convex semi-algebraic functions is a subclass of locally Lipschitz {\em nonsmooth} convex functions, and it includes first-order SDSOS-convex polynomial as a special case.

\begin{definition}[first-order SDSOS-convex semi-algebraic function]\label{SOS-SA}{\rm
A real function $f$ is called {\em first-order SDSOS-convex semi-algebraic} on $\mathbb{R}^{n}$ if it admits a representation
\begin{align*}
f(\bx) := \sup_{\by\in\Omega}\bigg\{h_{0}(\bx) +\sum^s_{j=1}y_jh_j(\bx)\bigg\},
\end{align*}
where
\begin{itemize}
\item[{\rm (i)}] each $h_j,$ $j = 0, 1,\ldots,s,$ is a polynomial, and for each $\by\in\Omega,$ $h_{0}+\sum^{s}_{j=1}y_jh_j$ is a first-order SDSOS-convex polynomial on $\mathbb{R}^{n};$
\item[{\rm (ii)}] $\Omega$ is a nonempty compact second-order cone program representable set given by
\begin{align*}
\Omega:= \left\{\by\in \mathbb{R}^s \colon \exists\, \bz\in \mathbb{R}^{p} \textrm{ s.t. } A_{0} +\sum^{s}_{j=1}y_{j}A_{j} +\sum^{p}_{\ell=1}z_{\ell}B_{\ell} \in S_{ sdd}^{q}\right\},
\end{align*}
%\textcolor[rgb]{1.00,0.00,0.00}{( All $n+1\Rightarrow q_i$.)}
with $A_{j},$ $j=0,1,\ldots,s,$ and $B_{\ell},$ $\ell=1,\ldots,p,$ being $q\times q$ symmetric matrices.
\end{itemize}
}
\end{definition}
	
\begin{remark}{\rm
\begin{enumerate}[\upshape (i)]
\item The main difference between Definition~\ref{SOS-SA} and Definition~\ref{sos-convex-sa} is that the set $\Omega$ in Definition~\ref{sos-convex-sa} is a nonempty compact {\em semidefinite program} representable set, while it is a nonempty compact {\em second-order cone program} representable set in Definition~\ref{SOS-SA}.
\item A first-order SDSOS-convex semi-algebraic function is an SOS-convex semi-algebraic function, but the converse is not true, as shown in the next Example~\ref{example-1}.
\end{enumerate}
}\end{remark}
\begin{example}\label{example-1}{\rm
Let  $f(\bx)=\sup\limits_{y \in \left[0,1 \right]}\left\{h_{0}(\bx) + y h_{1}(\bx)\right\}, \bx \in  \mathbb{R}^2,$ where $h_{0}(\bx)=x_{1}^{2}+2x_{2}^{2},\ h_{1}(\bx)=2x_{1}x_{2}.$
\begin{itemize}%[\upshape (a)]
\item  $f(\bx)$ is an SOS-convex semi-algebraic function.
To show this, we notice that
\begin{align*}
\Omega := & \left\{y \in \mathbb{R}\colon A =
\left(\begin{array}{ll}
0 & 0 \\
0 & 1
\end{array}\right)
+ y
\left(\begin{array}{ll}
1 & 0 \\
0 & -1
\end{array}\right) =
\begin{pmatrix}
y & 0 \\
0 & 1-y
\end{pmatrix}
\in S_{sdd}^{2}\right\} \\
= & \left\{y \in \mathbb{R} \colon \exists \, D= {\rm diag}\left ( d_1,d_2 \right ) \textrm{ s.t. } DAD \in S_{dd}^{2}, d_{i}>0,i=1,2\right\} \\
= & \left\{y \in \mathbb{R} \colon \exists \, d_{i}>0,i=1,2, \begin{pmatrix}
							d_{1}^{2}y & 0 \\
							0 & d_{2}^{2}(1-y)
						\end{pmatrix}  \in S_{dd}^{2} \right\} \\
= & \left\{y \in \mathbb{R} \colon 0 \leq y \leq 1 \right\}.
\end{align*}
Thus $\Omega=\left[0, 1\right]$ is a compact second-order cone program representable set and is also a compact semidefinite program representable set.
For any fixed $y \in \Omega,$ define $g_{y}(\bx):= h_{0}(\bx)+y h_{1}(\bx) = x_{1}^{2}+2x_{2}^{2}+2yx_{1}x_{2}.$
Then, we have
\begin{equation*}
\nabla g_{y}(\bx)=\binom{2x_{1}+2yx_{2}}{4x_{2}+2yx_{1}}; \quad
\nabla^{2}g_{y}(\bx)=\begin{pmatrix}
							2 & 2y \\
							2y & 4
						\end{pmatrix}.
\end{equation*}
Clearly, for each $y\in \Omega,$ $h_{0}(\bx) + y h_{1}(\bx)$ is an SOS-convex polynomial in variable $\bx$.
So, $f(\bx)$ is an SOS-convex semi-algebraic function.
\item  $f(\bx)$ is not a first-order SDSOS-convex semi-algebraic function.
To see this, we consider
\begin{align*}
h_g(\bx,\bz) &= g_{y}(\bx)-g_{y}(\bz)-\nabla g_{y}(\bz)^{T}(\bx-\bz)  \\
& = x_{1}^{2}+2x_{2}^{2}+z_{1}^{2}+2z_{2}^{2}+2yx_{1}x_{2}+2yz_{1}z_{2}-2yx_{1}z_{2}-2yx_{2}z_{1}-2x_{1}z_{1}-4x_{2}z_{2} \\
& = \begin{pmatrix}
							x_1 & x_2 & z_1 & z_2
						\end{pmatrix} \begin{pmatrix}
							1 & y & -1 & -y \\
							y & 2 & -y & -2 \\
							-1 & -y & 1 & y \\
							-y & -2 & y & 2
						\end{pmatrix}(:=Q) \begin{pmatrix}
							x_1 \\
							x_2 \\
							z_1\\
							z_2
						\end{pmatrix} .
\end{align*}
By contradiction, assume that $ h_g(\bx,\bz)$ is an SDSOS polynomial in variables $(\bx,\bz).$
Then there exists $D={\rm diag}\left (d_1, d_2, d_3, d_4 \right),\ d_{i} > 0,\ i = 1, 2, 3, 4, \textrm{ s.t. } DQD \in S_{dd}^{4}. $ By a simple calculation, we have
					\[
					\left\{
					\begin{array}{llll}
						d_1 \ \geq \ d_2y+d_3+d_4y,   \\
						2d_2 \ \geq \ d_1y+d_3y+2d_4,   \\
						d_3 \ \geq \ d_1+d_2y+d_4y,  \\
						2d_4 \ \geq \ d_1y+2d_2+d_3y,
					\end{array}
					\right. \ \Rightarrow \
					0 \geq d_1+d_2+d_3+d_4,
					\]
which contradicts the fact that $d_{i} > 0,$ for all $i = 1, 2, 3, 4.$
Thus $h_g(\bx, \bz)$ is not an SDSOS polynomial, so $f(\bx)$ is not a first-order SDSOS-convex semi-algebraic function.
\end{itemize}
}\end{example}
		
\begin{proposition}\label{proposition}
{\rm If $f_1$ and $f_2$ are first-order SDSOS-convex semi-algebraic functions on $\mathbb{R}^{n},$ then so is $f_1+f_2$.}
\end{proposition}	
\begin{proof}
As $f_{i}(\bx),\ i = 1, 2$ are first-order SDSOS-convex semi-algebraic functions, by definition, $ f_{i}(\bx) = \sup\limits_{\by^{i} \in \Omega_{i}}\left\{h_{0}^{i}(\bx) + \sum^{s_{i}}_{j=1}y_{j}^{i}h_{j}^{i}(\bx)\right\},$ and $\Omega_{i}$ is a nonempty compact second-order cone program representable set given by
\begin{align*}
\Omega_i = \left\{\by^{i} \in \mathbb{R}^{s_{i}} \colon \exists\, \bz^{i}\in \mathbb{R}^{p_{i}} \textrm{ s.t. } A_{0}^{i} +\sum^{s_{i}}_{j=1}y_{j}^{i}A_{j}^{i} +\sum^{p_{i}}_{\ell=1}z_{\ell}^{i}B_{\ell}^{i} \in S_{ sdd}^{q_i}\right\}.
\end{align*}
Then,
\begin{equation*}
f_{1}(\bx)+f_{2}(\bx)=\sup_{(\by^{1}, \by^{2}) \in \Omega_{1}\times\Omega_{2}}\left\{h_{0}^{1}(\bx)+h_{0}^{2}(\bx)+\sum_{j=1}^{s_{1}} y_{j}^{1} h_{j}^{1}(\bx)+\sum_{j=1}^{s_{2}} y_{j}^{2} h_{j}^{2}(\bx)\right\}.
\end{equation*}
Clearly, $h_{0}^{1}(\bx)+h_{0}^{2}(\bx)+\sum_{j=1}^{s_{1}} y_{j}^{1} h_{j}^{1}(\bx)+\sum_{j=1}^{s_{2}} y_{j}^{2} h_{j}^{2}(\bx)$ is a first-order SDSOS-convex polynomial on $\mathbb{R}^{n}.$
Note that $\Omega_{1} \times \Omega_{2}$ is also a nonempty compact second-order cone program representable set.
Thus, $f_1 + f_2$ is a first-order SDSOS-convex semi-algebraic function.
\end{proof}

\begin{example}\label{example-2}{\rm %\hfill\par
The class of first-order SDSOS-convex semi-algebraic functions contains many common nonsmooth convex functions.
Here we list several standard examples.
\begin{enumerate}[\upshape (a)]
\item Let $f(\bx) = \|\bx\|_2.$
Note that $f$ is a first-order SDSOS-convex semi-algebraic function, as it can be expressed as
\begin{align*}
\|\bx\|_2 = \max_{\by \in \Omega} \sum_{i=1}^{n} x_{i} y_{i},
\end{align*}
where $\Omega = \{\by\in\mathbb{R}^n \colon \|\by\|_2\leq1\}.$
As
\begin{align*}
\Omega & =\{\by\in\mathbb{R}^n \colon \|\by\|_2 \leq 1\} \\
& = \left\{ \by\in\mathbb{R}^n \colon M(\by):= I_{s+1}+\sum_{j=1}^{s}y_{j}A_{j}= \begin{pmatrix}
							1	& y_1 & \cdots & y_s \\
							y_1	& 1 & \cdots & 0 \\
							\vdots 	& \vdots  & \ddots  & \vdots  \\
							y_s	& 0 & \cdots  & 1
						\end{pmatrix} \succeq 0  \right\},
\end{align*}
where $I_{s+1}$ denotes the identity matrix of order $s+1,$ and $A_{j}$ is an $(s+1)\times (s+1)$ matrix whose elements satisfy:
\[
(A_j)_{i,k} =
\begin{cases}
1, & \text{if} \, \, i=1, k=j+1 \, \, \text{or} \, \, i=j+1,k=1, \\
0, & \text{otherwise}.
\end{cases}
\]
We now show $M(\by) \in S_{sdd}^{s+1}.$
Indeed, $D = \operatorname{diag}\left(1, \left | y_1 \right | , \ldots, \left | y_s \right |\right),$ then we have
\begin{equation*}
DM(\by)D=\begin{pmatrix}
1	& y_{1}\left | y_1 \right |  & \cdots & y_{s}\left | y_s \right |\\
y_{1}\left | y_1 \right |	& y_{1}^{2} & \cdots & 0 \\
\vdots	& \vdots & \ddots & \vdots \\
y_{s}\left | y_s \right |	& 0 & \cdots & y_{s}^{2}
\end{pmatrix}
\end{equation*}
is a diagonally dominant matrix.
Thus, $\Omega$ is a compact second-order cone program representable set.
Then, we can see
\begin{align*}
h_f(\bx,\bz) &=\sum_{i=1}^{n}x_iy_i -\sum_{i=1}^{n}z_iy_i-
\begin{pmatrix}
y_1, \cdots, y_n
\end{pmatrix}
\begin{pmatrix}
x_1-z_1 \\
\vdots \\
x_n-z_n
\end{pmatrix}  \\
& = \sum_{i=1}^{n}x_iy_i -\sum_{i=1}^{n}z_iy_i-\sum_{i=1}^{n}x_iy_i +\sum_{i=1}^{n}z_iy_i\\
& = 0.
\end{align*}
For any $\bx \in \mathbb{R}^{n},$ $\bz\in \mathbb{R}^{n},$ $h_f(\bx,\bz) \equiv 0,$ so $h_f$ is an SDSOS polynomial.

\item Let $f(\bx) = \|\bx\|_1.$
Observe that $f(\bx)$ can equivalently be expressed as $f(\bx) = \max\limits_{\by \in \Omega} \sum\limits_{i=1}^{n} x_{i} y_{i},$ where $\Omega = \left\{\by \in \mathbb{R}^{n} \colon \left|y_{i}\right| \leq 1, i = 1,\ldots, n \right\}.$
First, we show that $\Omega$ is a compact second-order cone program representable set.
To see this, for each component $y_i,$ we let
$M_{i}\left(y_{i}\right)=\left(\begin{array}{cc}
1 & y_{i} \\
y_{i} & 1
\end{array}\right),$
and then construct a block diagonal matrix:
\begin{align*}
M(\by)=\operatorname{diag}\left(M_{1}\left(y_{1}\right), M_{2}\left(y_{2}\right), \ldots, M_{n}\left(y_{n}\right)\right),
\end{align*}
which is asked to be an $sdd$ matrix, i.e., there is a diagonal matrix $D=\operatorname{diag}\left(d_{1}, d_{1}, \ldots, d_{n}, d_{n}\right)$ such that $D M(\by) D$ is a $dd$ matrix.
Computing each $2 \times 2$ block matrix yields that
\begin{align*}
D_{i} M_{i}\left(y_{i}\right) D_{i}&=\left(\begin{array}{cc}
d_{i}^{2} & d_{i}^{2} y_{i} \\
d_{i}^{2} y_{i} & d_{i}^{2}
\end{array}\right),
\end{align*}
since $D M(\by) D$ is a $dd$ matrix, by definition, it has $d_{i}^{2} \geq d_{i}^{2}\left|y_{i}\right|$ for each $i = 1, \ldots, n,$ and so $\left|y_{i}\right| \leq 1.$
Hence, $\Omega $ is a compact second-order cone program representable set.
It follows from the first-order SDSOS-convexity of $\sum\limits_{i = 1}^n x_{i} y_{i}$ that $f$ is a first-order SDSOS-convex semi-algebraic function.

\item Given a matrix $A \in \mathbb{R}^{m \times n}$ and a vector $b \in \mathbb{R}^{m},$ the least squares function is defined as
\begin{align*}
f(\bx) = \|A\bx - \bb\|_{2}^{2}.
\end{align*}
Along with Example~\ref{example-2} (a), $f(\bx)$ is a first-order SDSOS-convex semi-algebraic function. According to Proposition~\ref{proposition}, the least squares function with $\ell_{1}$-regularization  $$
f_{\mu}(\bx) = \|A \bx - \bb\|_{2}^{2} + \mu\|\bx\|_{1},
$$
where $\mu > 0,$ is a first-order SDSOS-convex semi-algebraic function.
\end{enumerate}
}\end{example}

\section{Exact SOCP Relaxation for first-order SDSOS-convex Semialgebraic Programs}\label{sect:4}

In this section, we focus on the study of the following optimization problems with first-order SDSOS-convex semi-algebraic data,
\begin{align}\label{P}
f_* := \min\limits_{\bx\in\mathbb{R}^{n}} \ \left\{f_0(\bx) \colon f_{i}(\bx)\leq0, \ i=1,\ldots,m \right\}, \tag{${\rm P}$}
\end{align}
where each $f_{i},$ $i = 0, 1, \ldots, m$ is a first-order SDSOS-convex semi-algebraic function of the form
\begin{align}\label{function f_{i}}
f_{i}(\bx) := \sup_{\by^{i} \in \Omega_{i}}\left\{h_{0}^{i}(\bx) +\sum^{s_{i}}_{j=1}y_{j}^{i}h_{j}^{i}(\bx)\right\},
\end{align}
such that
\begin{itemize}
\item[(i)] each $h_{j}^{i},$ $j = 0, \ldots, s_{i},$ $i = 0, 1, 2, \ldots, m$ is a polynomial with degree at most $2d,$ and for each $\by^{i}:=(y_{1}^{i},\ldots,y_{s_{i}}^{i})\in\Omega_{i},$ $h_{0}^{i}+\sum^{s_{i}}_{j=1}y_{j}^{i}h_{j}^{i}$ is a first-order SDSOS-convex polynomial on $\mathbb{R}^{n};$
\item[(ii)] each $\Omega_{i},$ $i = 0, 1, 2, \ldots, m$ is a nonempty compact second-order cone program representable set given by
\begin{align*}
\Omega_i = \left\{\by^{i} \in \mathbb{R}^{s_{i}} \colon \exists\, \bz^{i}\in \mathbb{R}^{p_{i}} \textrm{ s.t. } A_{0}^{i} +\sum^{s_{i}}_{j=1}y_{j}^{i}A_{j}^{i} +\sum^{p_{i}}_{\ell=1}z_{\ell}^{i}B_{\ell}^{i} \in S_{ sdd}^{q_i}\right\},
\end{align*}
where $A_{j}^{i},$ $j = 0, 1, \ldots, s_{i},$ and $B_{\ell}^{i},$ $\ell = 1, \ldots, p_{i},$ are $q_{i}\times q_{i}$ symmetric matrices.
\end{itemize}

Throughout this paper, we always assume that the feasible set $K:=\{\bx\in\mathbb{R}^{n}\colon f_{i}(\bx)\leq0, \ i = 1, \ldots, m\}$ of the problem~\eqref{P} is nonempty.
For convenience, we denote by
$$I:=\{0, 1, 2, \ldots, m\}.$$
Since for each $i = 0, 1, \ldots, m,$ the function $f_i$ as in \eqref{function f_{i}} belongs to a class of convex functions, so we can denote the set
\begin{align}\label{convexcone}
\mathcal{C}:=\bigcup_{\by^{i}\in\Omega_{i},\lambda_{i}\geq0}{\rm epi}\left(\sum_{i=1}^m\lambda_{i}\left(h_{0}^{i}+\sum_{j=1}^{s_{i}}y^{i}_{j}h^{i}_{j}\right)\right)^*,
\end{align}
which is a convex cone (see, e.g., \cite[Proposition~2.3]{Jeyakumar2010}).
Now, with the aid of the convex cone~\eqref{convexcone}, we first invoke the following Farkas-type lemma, which will play an important role in deriving our results.
The proof is essentially same as in \cite[Lemma 3.1]{Yang2024}, we therefore omit it.
\begin{lemma}\label{lemma6} %\textcolor[rgb]{1.00,0.00,0.00}{{\rm \cite{Yang2024}}}
Let $f_{i}: \mathbb{R}^{n} \to \mathbb{R},$ $i\in I,$ be convex functions defined as in \eqref{function f_{i}}.
Assume that the set $K := \{\bx\in\mathbb{R}^{n} \colon f_{i}(\bx) \leq 0, \ i = 1, \ldots, m\}$ is nonempty.
Then the following statements are equivalent$:$
\begin{itemize}
\item[{\rm (i)}] $K\subseteq \left\{\bx\in\mathbb{R}^{n} \colon f_{0}(\bx)\geq0 \right\};$
\item[{\rm (ii)}] $(0,0)\in {\rm epi\,}f_{0}^*+{\rm cl\,}\bigcup_{\by^{i}\in\Omega_{i},\lambda_{i}\geq0}{\rm epi}\left(\sum_{i=1}^m\lambda_{i}\left(h_{0}^{i}+\sum_{j=1}^{s_{i}}y^{i}_{j}h^{i}_{j}\right)\right)^*.$
\end{itemize}
\end{lemma}

\subsection{Computing the optimal value of the problem~\eqref{P}}\label{subsec-1}
We now formulate the SOCP relaxation dual problem for the problem~\eqref{P}.
\begin{align}\label{SDD}
\sup\limits_{\substack{\gamma, \lambda_{j}^{i},z_{\ell}^{i}}} \ \ \ & \gamma\tag{$\widehat{\rm Q}$}\\
\textrm{s.t.}\quad\   & \sum_{i=0}^{m}\left(\lambda_{0}^{i}h_{0}^{i} +\sum^{s_{i}}_{j=1}\lambda_{j}^{i}h_{j}^{i}\right)-\gamma\in\widetilde{\Sigma}[\bx]_{2d},\notag\\
&\lambda_{0}^{i}A_{0}^{i} +\sum^{s_{i}}_{j=1}\lambda_{j}^{i}A_{j}^{i} +\sum^{p_{i}}_{\ell=1}z_{\ell}^{i}B_{\ell}^{i} \in S_{sdd}^{q_i},\ i\in I,\notag\\
&\lambda_{0}^{0}=1, \ \lambda_{0}^{i}\geq0, \ \lambda_{j}^{i}\in\mathbb{R}, \ j=1,\ldots, s_{i}, \ z_{\ell}^{i}\in\mathbb{R}, \  \ell=1,\ldots,p_{i}, \  i=1,\ldots,m.\notag
\end{align}
It is worth noting that the problem~\eqref{SDD} is an SOCP problem; moreover, it can be efficiently solved via interior point methods.

\medskip
In what follows, we propose the strong duality result between the primal problem~\eqref{P} and its dual problem~\eqref{SDD}; apart from this, we also need the following assumption.
\begin{itemize}%\label{assumption1}
\item[{\bf (A1)}] The convex cone
$\mathcal{C}$ as in \eqref{convexcone}
is closed.
\end{itemize}
\begin{remark}{\rm
It is worth mentioning that if the {\it Slater condition} holds for \eqref{P}, that is, there exists $\widehat{\bx}\in\mathbb{R}^{n}$ such that
\begin{align*}
f_{i}(\widehat{\bx}) < 0, \ i = 1, \ldots, m,
\end{align*}
then the convex cone $\mathcal{C}$ is closed (see, e.g., \cite[Proposition~3.2]{Jeyakumar2010}).
}
\end{remark}

\begin{theorem}[\textbf{Strong Duality Theorem}]\label{thm1}
Suppose that assumption {\bf (A1)} holds.
Then
\begin{align*}
\inf\eqref{P}=\max\eqref{SDD}.
\end{align*}
\end{theorem}

\begin{proof}
We first show that $\inf\eqref{P}\leq \max\eqref{SDD}.$
Let $\overline{\gamma}:=\inf\eqref{P}\in\mathbb{R},$ so $f_{0}(\bx)-\overline{\gamma}\geq0$ for all $\bx\in K.$
It then follows from assumption {\bf (A1)} and Lemma~\ref{lemma6} that there exist $\overline{\lambda}_{i}\geq0$ and $\overline{\by}^{i}\in\Omega_{i},$ $i = 1, \ldots, m$ such that
\begin{align*}
(0, -\overline{\gamma})\in {\rm epi\,}f_{0}^* + {\rm epi}\left(\sum_{i=1}^m\overline{\lambda}_{i}\left(h_{0}^{i}+\sum_{j=1}^{s_{i}}\overline{y}^{i}_{j}h^{i}_{j}\right)\right)^*.
\end{align*}
Then, there exist $(\bxi,\alpha)\in{\rm epi\,}f_{0}^*$ and $(\bzeta,\beta)\in{\rm epi}\left(\sum\limits_{i=1}^m\overline{\lambda}_{i}\left(h_{0}^{i}+\sum\limits_{j=1}^{s_{i}}\overline{y}^{i}_{j}h^{i}_{j}\right)\right)^*$ such that
$(0,-\overline{\gamma})=(\bxi,\alpha)+(\bzeta,\beta),$
and so, for each $\bx\in\mathbb{R}^{n},$ we have
\begin{align*}
	&\langle \bxi,\bx\rangle-f_{0}(\bx)+\langle \bzeta,\bx\rangle-\sum_{i=1}^m\overline{\lambda}_{i}\left(h_{0}^{i}(\bx)+\sum_{j=1}^{s_{i}}\overline{y}^{i}_{j}h^{i}_{j}(\bx)\right)\\
	\leq\ &f_{0}^*(\bxi)+\left(\sum_{i=1}^m\overline{\lambda}_{i}\left(h_{0}^{i}+\sum_{j=1}^{s_{i}}\overline{y}^{i}_{j}h^{i}_{j}\right)\right)^*(\bzeta)\\
	\leq\ & \alpha+\beta=-\overline{\gamma},
\end{align*}
i.e., for each $\bx\in \mathbb{R}^{n},$
\begin{align}\label{thm1rel3}
	f_{0}(\bx)+\sum_{i=1}^m\overline{\lambda}_{i}\left(h_{0}^{i}(\bx)+\sum_{j=1}^{s_{i}}\overline{y}^{i}_{j}h^{i}_{j}(\bx)\right)-\overline{\gamma}\geq0.
\end{align}
Note that $f_{0}$ is defined as $f_{0}(\bx) = \sup_{\by^{0}\in\Omega_{0}}\{h_{0}^{0}(\bx) +\sum^{s_{0}}_{j=1}y_{j}^{0}h_{j}^{0}(\bx)\}.$
Since $\Omega_{0}$ is a compact set, there exists $\overline{\by}^{0}\in\Omega_{0}$ such that
\begin{align*}
	f_{0}(\bx)=h_{0}^{0}(\bx)+\sum_{j=1}^{s_{0}}\overline{y}^{0}_{j}h^{0}_{j}(\bx).
\end{align*}
It then follows from \eqref{thm1rel3} that
\begin{align*}
	h_{0}^{0}(\bx)+\sum_{j=1}^{s_{0}}\overline{y}^{0}_{j}h^{0}_{j}(\bx)+\sum_{i=1}^m\overline{\lambda}_{i}\left(h_{0}^{i}(\bx)+\sum_{j=1}^{s_{i}}\overline{y}^{i}_{j}h^{i}_{j}(\bx)\right)-\overline{\gamma}\geq0
\end{align*}
for all $\bx\in\mathbb{R}^n.$
Let $\Phi\colon\R^n\to\R$ be defined by
\begin{align*}
	\Phi(\bx):=h_{0}^{0}(\bx)+\sum_{j=1}^{s_{0}}\overline{y}^{0}_{j}h^{0}_{j}(\bx)+\sum_{i=1}^m\overline{\lambda}_{i}\left(h_{0}^{i}(\bx)+\sum_{j=1}^{s_{i}}\overline{y}^{i}_{j}h^{i}_{j}(\bx)\right)-\overline{\gamma}.
\end{align*}
Since $h_{0}^{i}+\sum_{j=1}^{s_{i}}\overline{y}^{i}_{j}h^{i}_{j},$ $i\in I,$ are first-order SDSOS-convex polynomials and $\lambda_{i}\geq0,$ $i=1,\ldots,m,$ we see that $\Phi$ is also a first-order SDSOS-convex polynomial.
Moreover, since $\Phi$ is nonnegative on $\mathbb{R}^n,$ it follows from Lemma~\ref{lemma4} that $\Phi$ is SDSOS, i.e,
\begin{align}\label{thm1rel4}
	h_{0}^{0}+\sum_{j=1}^{s_{0}}\overline{y}^{0}_{j}h^{0}_{j}+\sum_{i=1}^m\overline{\lambda}_{i}\left(h_{0}^{i}+\sum_{j=1}^{s_{i}}\overline{y}^{i}_{j}h^{i}_{j}\right)-\overline{\gamma} \in \widetilde{\Sigma}[\bx]_{2d}.
\end{align}
On the other hand, for each $i\in I,$ as $\overline{\by}^{i}=(\overline{y}^{i}_{1},\ldots,\overline{y}^{i}_{s_{i}})\in\Omega_{i},$ there exists $\overline{\bz}^{i}:=(\overline{z}^{i}_{1},\ldots,\overline{z}^{i}_{p_{i}})\in \mathbb{R}^{p_{i}}$ such that
\begin{align}\label{thm1rel5}
	A_{0}^{i}+\sum_{j=1}^{s_{i}}\overline{y}_{j}^{i}A_{j}^{i}+\sum_{\ell=1}^{p_{i}}\overline{z}_{\ell}^{i}B_{\ell}^{i}\in S_{sdd}^{q_i}, \ i\in I.
\end{align}
For each $i=1,\ldots,m,$ let $\lambda_{0}^{i}:=\overline{\lambda}_{i},$ $\lambda_{j}^{i}:=\overline{\lambda}_{i}\overline{y}_{j}^{i},$ and $z_{\ell}^{i}:=\overline{\lambda}_{i}\overline{z}_{\ell}^{i},$  $j=1,\ldots,s_{i},$ $\ell=1,\ldots,p_{i}.$
Let $\lambda_{0}^{0}:=1,$ $\lambda_{j}^{0}:=\overline{y}_{j}^{0},$ $j=1,\ldots,s_{0},$ and $z_{\ell}^{0}:=\overline{z}_{\ell}^{0},$ $\ell=1,\ldots,s_{0}.$
Then, from \eqref{thm1rel4}, we have
\begin{align*}%\label{thm1rel6}
	\sum_{i=0}^{m}\left(\lambda_{0}^{i}h_{0}^{i} +\sum^{s_{i}}_{j=1}\lambda_{j}^{i}h_{j}^{i}\right)-\overline{\gamma}\in\widetilde{\Sigma}[\bx]_{2d}.
\end{align*}
Note that $\overline{\lambda}_{i},$ $i=1,\ldots,m,$ are nonnegative.
It follows from \eqref{thm1rel5} that
\begin{align*}
	\lambda_{0}^{i}A_{0}^{i} +\sum^{s_{i}}_{j=1}\lambda_{j}^{i}A_{j}^{i} +\sum^{p_{i}}_{\ell=1}z_{\ell}^{i}B_{\ell}^{i}\in S_{ sdd}^{q_i}, \ i\in I.
\end{align*}
It means that the tuple $(                       \overline{\gamma},\blambda^{0},\ldots,\blambda^m,\bz^{0},\ldots,\bz^m)$ is a feasible solution of the problem~\eqref{SDD}, where for each $i\in I,$ $\blambda^{i}:=(\lambda^{i}_{0},\lambda^{i}_{1},\ldots,\lambda^{i}_{s_{i}})$ and $\bz^{i}:=(z^{i}_{1},\ldots,z^{i}_{p_{i}}).$
This fact leads to
\begin{align*}
	\inf\eqref{P}=\overline{\gamma}\leq\max\eqref{SDD}.
\end{align*}

Next, we claim that $\inf\eqref{P}\geq\max\eqref{SDD}.$
Let $(\gamma,\blambda^{0},\ldots,\blambda^m,\bz^{0},\ldots,\bz^m)$ be any feasible solution of the problem~\eqref{SDD}.
Then, we have
\begin{align}
	&\sum_{i=0}^{m}\left(\lambda_{0}^{i}h_{0}^{i} +\sum^{s_{i}}_{j=1}\lambda_{j}^{i}h_{j}^{i}\right)-\gamma\in\widetilde{\Sigma}[\bx]_{2d},\label{thm1rel1}\\
	&\lambda_{0}^{i}A_{0}^{i} +\sum^{s_{i}}_{j=1}\lambda_{j}^{i}A_{j}^{i} +\sum^{p_{i}}_{\ell=1}z_{\ell}^{i}B_{\ell}^{i} \in S_{ sdd}^{q_i},\ i\in I\notag .
\end{align}
For each $i\in I,$ pick $\overline{\by}^{i}:=(\overline{y}^{i}_{1},\ldots,\overline{y}^{i}_{s_{i}})\in\Omega_{i}.$
Then, by the definition of $\Omega_{i},$ there exist $\overline{\bz}^{i}:=(\overline{z}^{i}_{1},\ldots,\overline{z}^{i}_{p_{i}})\in \mathbb{R}^{p_{i}},$ $i\in I,$ such that $A^{i}_{0}+\sum^{s_{i}}_{j=1}\overline y_{j}^{i}A_{j}^{i}+\sum^{p_{i}}_{\ell=1}\overline z_{\ell}^{i}B_{\ell}^{i}\in S_{ sdd}^{q_i}.$
Now, for each $i\in I$ and each $j=1,\ldots,s_{i},$ put
\begin{align*}
	\widetilde{y}_{j}^{i}:=
	\left\{
	\begin{array}{ll}
		\frac{\lambda_{j}^{i}}{\lambda_{0}^{i}}, & \textrm{if } \lambda_{0}^{i}>0, \\
		\overline{y}_{j}^{i}, & \textrm{if } \lambda_{0}^{i}=0,
	\end{array}
	\right.
	\quad \textrm{ and } \quad
	\widetilde{z}_{\ell}^{i}:=
	\left\{
	\begin{array}{ll}
		\frac{z_{\ell}^{i}}{\lambda_{0}^{i}}, & \textrm{if } \lambda_{0}^{i}>0, \\
		\overline{z}_{\ell}^{i}, & \textrm{if } \lambda_{0}^{i}=0.
	\end{array}
	\right.
\end{align*}
Note that $\lambda_{0}^{0}=1.$
Then, for each $i\in I,$ we have
\begin{align*}
	A_{0}^{i}+\sum_{j=1}^{s_{i}}\widetilde{y}^{i}_{j}A_{j}^{i}+\sum_{\ell=1}^{p_{i}}\widetilde{z}^{i}_{j}B_{\ell}^{i}=
	\left\{
	\begin{array}{ll}
		\displaystyle \frac{1}{\lambda_{0}^{i}}\left(\lambda_{0}^{i}A_{0}^{i}+\sum_{j=1}^{s_{i}}\lambda_{j}^{i}A_{j}^{i}+\sum_{\ell=1}^{p_{i}}z_{\ell}^{i}B_{\ell}^{i}\right), & \textrm{if } \lambda_{0}^{i}>0, \\
		\displaystyle A_{0}^{i}+\sum_{j=1}^{s_{i}}\overline{y}_{j}^{i}A_{j}^{i}+\sum_{\ell=1}^{p_{i}}\overline{z}_{\ell}^{i}B_{\ell}^{i}, & \textrm{if } \lambda_{0}^{i}=0,
	\end{array}
	\right.
\end{align*}
which is a scaled diagonally dominant matrix.
So, we see that $\widetilde{\by}^{i}:=(\widetilde{y}^{i}_{1},\ldots,\widetilde{y}^{i}_{s_{i}})\in\Omega_{i},$ $i\in I.$
Note that each set $\Omega_{i}$ is compact.
Note that for each $i\in I,$ if $\lambda_{0}^{i}=0,$ then
$\lambda_{j}^{i}=0$ for all $j=1,\ldots,s_{i}$ (see, e.g., \cite[Lemma 4.1]{Chieu2018}).
This implies that for each $i\in I,$
\begin{align*}
	\lambda_{0}^{i}\left(h_{0}^{i}(\bx)+\sum^{s_{i}}_{j=1}\widetilde{y}_{j}^{i}h_{j}^{i}(\bx)\right)=\lambda_{0}^{i}h_{0}^{i}(\bx)+\sum^{s_{i}}_{j=1}\lambda_{0}^{i}\widetilde{y}_{j}^{i}h_{j}^{i}(\bx)=\lambda_{0}^{i}h_{0}^{i}(\bx)+\sum_{j=1}^{s_{i}}\lambda_{j}^{i}h_{j}^{i}(\bx).
\end{align*}
This, together with \eqref{thm1rel1}, yields that for any $\bx\in K,$
\begin{align*}
	f_{0}(\bx)\geq \lambda_{0}^{0}\left(h_{0}^{0}(\bx)+\sum^{s_{0}}_{j=1}\widetilde{y}_{j}^{0}h_{j}^{0}(\bx)\right)\geq\sum_{i=0}^{m}\left(\lambda_{0}^{i}h_{0}^{i}(\bx) +\sum^{s_{i}}_{j=1}\lambda_{j}^{i}h_{j}^{i}(\bx)\right)\geq\gamma,
\end{align*}
i.e., $f_{0}(\bx)\geq\gamma$ for all $\bx\in K.$
Since $(\gamma,\blambda^{0},\ldots,\blambda^m,\bz^{0},\ldots,\bz^m)$ is arbitrary, we have
\begin{align*}
	\inf\eqref{P}\geq \sup\eqref{SDD},
\end{align*}
and thus, the proof is complete.
\end{proof}

\begin{remark}{\rm
As mentioned above, if the Slater condition holds for \eqref{P}, then assumption~{\bf (A1)} is also satisfied.
As a result, Theorem~\ref{thm1} still holds true under the Slater condition.
}\end{remark}

\subsection{Recovering optimal solutions}
By solving the problem~\eqref{SDD}, we can obtain the optimal value of the primal problem~\eqref{P} as shown in  subsection~\ref{subsec-1}.
However, there is no information on optimal solutions from the problem~\eqref{SDD}, in order to recover an optimal solution, we formulate the following Lagrangian dual problem of the problem~\eqref{SDD},
\begin{align}\label{SDP}
\inf_{\substack{\bw\in\mathbb{R}^{s(n,2d)}\\Z_{i}\in S^{q_i},i\in I}} \ & \sum_{\ba\in\mathbb{N}^{n}_{2d}} \left(h_{0}^{0}\right)_{\ba} w_{\ba}+\left\langle A_{0}^{0},Z_{0}\right\rangle \tag{${\rm Q}$}\\
\text{s.t.}\quad\,\ &  \sum_{\ba\in\mathbb{N}_{2d}^n} \left(h_{0}^{i}\right)_{\ba} w_{\ba}+\left\langle A_{0}^{i},Z_{i} \right\rangle\leq0, \ i=1,\ldots,m,\notag\\
& \sum_{\ba\in\mathbb{N}_{2d}^n} \left(h_{j}^{i}\right)_{\ba} w_{\ba}+\left\langle A_{j}^{i}, Z_{i}\right\rangle=0, \ i\in I, \ j=1,\ldots, s_{i},\notag\\
& \left\langle B_{\ell}^{i}, Z_{i} \right\rangle=0, \ i\in I, \ \ell=1,\ldots,p_{i},\notag\\
&  \left\|\left(
\begin{array}{c}
	2(Z_{i})_{rs} \\
	(Z_{i})_{rr} - (Z_{i})_{ss} \\
\end{array}
\right)\right\|\leq (Z_{i})_{rr}+(Z_{i})_{ss}, \ 1\leq r<s\leq q_i, \ i\in I, \nonumber \\
&  \left\|\left(
\begin{array}{c}
	2(\M_d(\bw))_{ij} \\
	(\M_d(\bw))_{ii} - (\M_d(\bw))_{jj} \\
\end{array}
\right)\right\|\leq (\M_d(\bw))_{ii}+(\M_d(\bw))_{jj}, \ 1\leq i<j\leq s(n,d), \nonumber \\
&  w_{\bze}=1.\nonumber
\end{align}
It is also worth noting that the problem~\eqref{SDP} is an SOCP problem, which can be solved efficiently.

We now give the following result, which can extract an optimal solution to the problem~\eqref{P} by solving the SOCP problem~\eqref{SDP}, and the proof is motivated by~\cite[Theorem 4.2]{Chieu2018}.
For this, the following assumption is also needed.
\begin{itemize}%\label{assumption2}
\item[{\bf (A2)}] For each $i\in I,$ there exist $\widehat{\by}^{i}\in\mathbb{R}^{s_{i}},$ $\widehat{\bz}^{i}\in\mathbb{R}^{p_{i}}$ and a diagonal matrix $D_{i}$ with all positive diagonal entries such that $D^{i}\left(A_{0}^{i}+\sum_{j=1}^{s_{i}} \widehat{y}_{j}^{i} A_{j}^{i}+\sum_{\ell=1}^{p_{i}} \widehat{z}_{\ell}^{i} B_{\ell}^{i}\right) D^{i}$ is strictly diagonally dominant.
\end{itemize}

\begin{theorem}\label{thm2}
Suppose that assumptions~{\bf (A1)} and~{\bf (A2)} hold.
Then$,$ we have
\begin{align*}
\inf\eqref{P}=\max\eqref{SDD} =\inf\eqref{SDP}.
\end{align*}
We further assume that $(\overline{\bw}, \overline{Z}_{0},\ldots,\overline{Z}_{m})$ is an optimal solution to the problem~\eqref{SDP}.
Then
\begin{align*}
\overline{\bx}:=\left(L_{\overline{\bw}}(x_{1}),\ldots,L_{\overline{\bw}}(x_{n})\right)
\end{align*}
is an optimal solution to the problem~\eqref{P}.
\end{theorem}
\begin{proof}
Note that $\inf\eqref{P}=\max\eqref{SDD}$ holds by Theorem~\ref{thm1} (under assumption~{\bf (A1)}).
Since the problem~\eqref{SDD} is the Lagrangian dual problem\footnote{Indeed, \eqref{SDD} and \eqref{SDP} are dual problems to each other; see the proof of \cite[Theorem 4.2]{Chieu2018}.} of the problem~\eqref{SDP}, by the usual weak duality, we see that $\inf\eqref{SDP}\geq \max\eqref{SDD}$ holds.

To prove the second equality in the first argument, we now show that $\inf\eqref{P}\geq \inf\eqref{SDP}.$
To see this, let $\bx$ be any feasible solution to the problem~\eqref{P}, and let $\gamma:=f_{0}(\bx).$
Then, we have $f_i(\bx)\leq0,$ $i = 1, \ldots, m.$
Note that
\begin{align*}
f_{i}(\bx) = \sup_{\by^i \in \Omega_{i}}\left\{h_{0}^{i}(\bx) +\sum^{s_{i}}_{j=1}y_{j}^{i}h_{j}^{i}(\bx)\right\}, \ i\in I.
\end{align*}
So, by the definition of each $f_i,$ we see that
\begin{align*}
&h_{0}^{0}(\bx) +\sum^{s_{0}}_{j=1}y_{j}^{0}h_{j}^{0}(\bx)\leq \gamma,\\
&h_{0}^{i}(\bx) +\sum^{s_{i}}_{j=1}y_{j}^{i}h_{j}^{i}(\bx)\leq f_{i}(\bx)\leq0, \ i=1,\ldots,m,
\end{align*}
for all $\by^i \in \Omega_{i},$ $i\in I,$ where each compact set $\Omega_{i}$ is given by
\begin{align*}
\Omega_{i} = \left\{(y_{1}^{i},\ldots,y_{s_{i}}^{i})\in \mathbb{R}^{s_{i}} \colon \exists \, \bz^{i}\in \mathbb{R}^{p_{i}} \textrm{ s.t. } A_{0}^{i} +\sum^{s_{i}}_{j=1}y_{j}^{i}A_{j}^{i} +\sum^{p_{i}}_{\ell=1}z_{\ell}^{i}B_{\ell}^{i} \in S_{sdd}^{q_i}\right\}.
\end{align*}
%\textcolor[rgb]{1.00,0.00,0.00}{Since the SDD cone is a proper cone, that is, it is convex, closed, pointed, and has a nonempty interior,  then, by Assumption~\ref{assumption2} and the strong duality theorem for  for conic programming \cite[Section~5.9]{Boyd2004}, there exist $W_{i}\succeq0,$ $i\in I,$ such that}
Since assumption~{\bf (A2)} holds, it follows from the strong duality theorem for conic programming \cite[Section~5.9]{Boyd2004} that there exist $Z_i\in S^{q_i},$ $i\in I,$ such that
\begin{align}\label{thm2rel1}
	\left\{
	\begin{array}{l}
		h_{0}^{0}(\bx)+\langle A_{0}^{0},Z_{0}\rangle \leq {\gamma},\\
		h_{0}^{i}(\bx)+\langle A_{0}^{i},Z_{i}\rangle\leq0, \ i=1,\ldots,m,\\
		h_{j}^{i}(\bx)+\langle A_{j}^{i},Z_{i}\rangle=0, \  i\in I, \ j=1,\ldots,s_{i},\\
		\langle B_{\ell}^{i},Z_{i}\rangle=0, \  i\in I, \ \ell=1,\ldots,p_{i},\\
		\left\|\left(
		\begin{array}{c}
			2(Z_{i})_{rs} \\
			(Z_{i})_{rr} - (Z_{i})_{ss} \\
		\end{array}
		\right)\right\|\leq (Z_{i})_{rr}+(Z_{i})_{ss}, \ 1\leq r<s\leq q_i, \ i\in I. \\
	\end{array}
	\right.
\end{align}
Letting $\bw:=\lceil \bx\rceil_d\in\mathbb{R}^{s(n,2d)},$ \eqref{thm2rel1} can be rewritten as follows:
\begin{align*}
	\left\{
	\begin{array}{l}
		\sum\limits_{\ba\in\mathbb{N}_{2d}^n}(h_{0}^{0})_{\ba} w_{\ba}+\langle A_{0}^{0},Z_{0}\rangle \leq\gamma,\\
		\sum\limits_{\ba\in\mathbb{N}_{2d}^n}(h_{0}^{i})_{\ba} w_{\ba}+\langle A_{0}^{i},Z_{i}\rangle\leq0, \ i=1,\ldots,m,\\
		\sum\limits_{\ba\in\mathbb{N}_{2d}^n}(h_{j}^{i})_{\ba} w_{\ba}+\langle A_{j}^{i},Z_{i}\rangle=0, \  i\in I, \ j=1,\ldots,s_{i},\\
		\langle B_{\ell}^{i},Z_{i}\rangle=0, \  i\in I, \ \ell=1,\ldots,p_{i},\\
		\left\|\left(
		\begin{array}{c}
			2(Z_{i})_{rs} \\
			(Z_{i})_{rr} - (Z_{i})_{ss} \\
		\end{array}
		\right)\right\|\leq (Z_{i})_{rr}+(Z_{i})_{ss}, \ 1\leq r<s\leq q_i, \ i\in I. \\
	\end{array}
	\right.
\end{align*}
In addition, from the definition of the moment matrix we see that
\begin{align*}
\M_d(\bw)=\sum_{\ba=1}^{s(n, d)}w_{\ba}\left(\lceil \bx\rceil_{d}\lceil \bx\rceil_{d}^{T}\right)_{\ba}=\lceil \bx\rceil_{d} \lceil \bx\rceil_{d}^{T} \succeq 0 ,
\end{align*}
which implies that
\begin{align*}
	\left\|\left(
	\begin{array}{cc}
		2(\M_d(\bw))_{ij} \\
		(\M_d(\bw))_{ii} - (\M_d(\bw))_{jj} \\
	\end{array}
	\right)\right\|\leq (\M_d(\bw))_{ii}+(\M_d(\bw))_{jj}, \ 1\leq i<j\leq s(n,d),
\end{align*}
and so, $(\bw, Z_{0},\ldots, Z_{m})$ is a feasible solution to the problem~\eqref{SDP}, and hence,
\begin{align*}
f_{0}(\bx)=\gamma\geq \sum\limits_{\ba\in\mathbb{N}_{2d}^n}(h_{0}^{0})_{\ba} w_{\ba}+\langle A_{0}^{0},Z_{0}\rangle\geq \inf\eqref{SDP}.
\end{align*}
Since $\bx\in K$ is arbitrary, we have $\inf\eqref{P}\geq \inf\eqref{SDP}.$

To prove the second argument, let $(\overline{\bw}, \overline{Z}_{0},\ldots,\overline{Z}_{m})$ be an optimal solution to the problem~\eqref{SDP}.
Then, we have
\begin{align}
&\sum\limits_{\ba\in\mathbb{N}_{2d}^n}(h_{0}^{i})_{\ba} \overline{w}_{\ba}+\langle A_{0}^{i},\overline{Z}_{i}\rangle\leq0, \ i=1,\ldots,m,\label{thm2rel2}\\
&\sum\limits_{\ba\in\mathbb{N}_{2d}^n}(h_{j}^{i})_{\ba} \overline{w}_{\ba}+\langle A_{j}^{i},\overline{Z}_{i}\rangle=0, \  i\in I, \ j=1,\ldots,s_{i},\label{thm2rel3}\\
&\langle B_{\ell}^{i},\overline{Z}_{i}\rangle=0, \  i\in I, \ \ell=1,\ldots,p_{i},\label{thm2rel4}\\
&\left\|\left(
	\begin{array}{c}
		2(\overline{Z}_{i})_{rs} \\
		(\overline{Z}_{i})_{rr} - (\overline{Z}_{i})_{ss} \\
	\end{array}
	\right)\right\|\leq (\overline{Z}_{i})_{rr}+(\overline{Z}_{i})_{ss}, \ 1\leq r<s\leq q_i, \ i\in I, \notag \\%\label{thm2rel5}\\
	&  \left\|\left(
	\begin{array}{cc}
		2(\M_d(\overline{\bw}))_{ij} \\
		(\M_d(\overline{\bw}))_{ii} - (\M_d(\overline{\bw}))_{jj} \\
	\end{array}
	\right)\right\|\leq (\M_d(\overline{\bw}))_{ii}+(\M_d(\overline{\bw}))_{jj}, \ 1\leq i<j\leq s(n,d), \ \overline{w}_{\bze}=1.\label{thm2rel6}
\end{align}
Let $(y_{1}^{i},\ldots,y_{s_{i}}^{i})\in\Omega_{i},$ $i\in I,$ be any given.
Then, for each $i\in I,$ there exists $(z_{1}^{i},\ldots,z_{p_{i}}^{i})\in\mathbb{R}^{p_{i}}$ such that
\begin{align*}
A_{0}^{i} +\sum^{s_{i}}_{j=1}y_{j}^{i}A_{j}^{i} +\sum^{p_{i}}_{\ell=1}z_{\ell}^{i}B_{\ell}^{i} \in S_{sdd}^{q_i}.
\end{align*}
It follows from (\ref{thm2rel2}) that for each $i=1,\ldots,m,$
	\begin{align}
		0\geq\sum\limits_{\ba\in\mathbb{N}_{2d}^n}(h_{0}^{i})_{\ba} \overline{w}_{\ba}+\langle A_{0}^{i},\overline{Z}_{i}\rangle
		\geq \ & \sum\limits_{\ba\in\mathbb{N}_{2d}^n}(h_{0}^{i})_{\ba} \overline{w}_{\ba}+\langle A_{0}^{i},\overline{Z}_{i}\rangle- \left \langle A_{0}^{i} +\sum^{s_{i}}_{j=1}y_{j}^{i}A_{j}^{i} +\sum^{p_{i}}_{\ell=1}z_{\ell}^{i}B_{\ell}^{i} , \overline{Z}_{i}\right \rangle \notag\\
		=\ &\sum\limits_{\ba\in\mathbb{N}_{2d}^n}(h_{0}^{i})_{\ba} \overline{w}_{\ba} -\sum^{s_{i}}_{j=1}y_{j}^{i}\langle A_{j}^{i},\overline{Z}_{i}\rangle -\sum^{p_{i}}_{\ell=1}z_{\ell}^{i}\langle B_{\ell}^{i},\overline{Z}_{i}\rangle\notag\\
		= \ &\sum\limits_{\ba\in\mathbb{N}_{2d}^n}(h_{0}^{i})_{\ba} \overline{w}_{\ba}+\sum^{s_{i}}_{j=1}y_{j}^{i}\sum\limits_{\ba\in\mathbb{N}_{2d}^n}(h_{j}^{i})_{\ba} \overline{w}_{\ba} \notag\\
		= \ &L_{\overline{\bw}}\left(h_{0}^{i}+\sum_{j=1}^{s_{i}}y_{j}^{i}h_{j}^{i}\right),\label{thm2rel7}
	\end{align}
	where the second equality follows from (\ref{thm2rel3}) and (\ref{thm2rel4}).
Note that for each $i=1,\ldots,m,$ $h_{0}^{i}+\sum_{j=1}^{s_{i}}y_{j}^{i}h_{j}^{i}$ is a first-order SDSOS-convex polynomial.
Since $\overline{\bw}$ satisfies (\ref{thm2rel6}), by Lemma~\ref{lemma5}, we see that for each $i=1,\ldots,m,$
\begin{align}\label{thm2rel9}
	L_{\overline{\bw}}\left(h_{0}^{i}+\sum_{j=1}^{s_{i}}y_{j}^{i}h_{j}^{i}\right)\geq \left(h_{0}^{i}+\sum_{j=1}^{s_{i}}y_{j}^{i}h_{j}^{i}\right)(L_{\overline{\bw}}(x_{1}),\ldots,L_{\overline{\bw}}(x_{n}))=h_{0}^{i}(\overline{\bx})+\sum_{j=1}^{s_{i}}y_{j}^{i}h_{j}^{i}(\overline{\bx}).
\end{align}
This, together with (\ref{thm2rel7}), yields that for each $i=1,\ldots,m,$
\begin{align*}%\label{thm2rel8}
	0\geq \sup_{\by^{i}\in\Omega_{i}}\left\{h_{0}^{i}(\overline{\bx})+\sum_{j=1}^{s_{i}}y_{j}^{i}h_{j}^{i}(\overline{\bx})\right\}=f_{i}(\overline{\bx})
\end{align*}
as each $\by^{i}\in\Omega_{i}$ is arbitrary.
So, $\overline{\bx}$ is a feasible solution to the problem~\eqref{P}.
Moreover, by similar arguments as above into \eqref{thm2rel9}, we have
\begin{align*}
	L_{\overline{\bw}}\left(h_{0}^{0}+\sum_{j=1}^{s_{0}}y_{j}^{0}h_{j}^{0}\right)\geq \left(h_{0}^{0}+\sum_{j=1}^{s_{0}}y_{j}^{0}h_{j}^{0}\right)(L_{\overline{\bw}}(x_{1}),\ldots,L_{\overline{\bw}}(x_{n}))=h_{0}^{0}(\overline{\bx})+\sum_{j=1}^{s_{0}}y_{j}^{0}h_{j}^{0}(\overline{\bx}).
\end{align*}
This implies that
\begin{align*}
	\inf\eqref{SDP}=\sum_{\ba\in\mathbb{N}^{n}_{2d}} \left(h_{0}^{0}\right)_{\ba} \overline{w}_{\ba}+\left\langle A_{0}^{0},\overline{Z}_{0}\right\rangle
	\geq& \ L_{\overline{\bw}}\left(h_{0}^{0}+\sum_{j=1}^{s_{0}}y_{j}^{i}h_{j}^{0}\right)\geq \ f_{0}(\overline{\bx})\geq \ \inf\eqref{P}=\inf\eqref{SDP}.
\end{align*}
Thus, $\overline{\bx}:=(L_{\overline{\bw}}(x_{1}),\ldots,L_{\overline{\bw}}(x_{n}))$ is an optimal solution to the problem~\eqref{P}.
\end{proof}

\subsection{An illustrative example}
We finish this section by designing a simple example, which will illustrate how to find an optimal solution to the problem~\eqref{P} by using solving an SOCP problem.
\begin{example}{\rm (see \cite[Example 4.1]{Chieu2018}.)
Consider the following $2$-dimensional first-order SDSOS-convex semi-algebraic optimiazation problem:
\begin{align}\label{P1}
\min\limits_{\bx\in\mathbb{R}^{2}}\ \  & x_{1}^{4}-x_{2} \tag{${{\rm P}_{1}}$}\\
\textrm{s.t.}\ \ \,&  x_{1}^{2}+x_{2}^{2}+2\left\|\left(x_{1}, x_{2}\right)\right\|_{2}-1\leq0. \notag
\end{align}
In order to match the notation used in the problem~\eqref{P}.
Let $h_{0}^{0}(\bx)=x_{1}^{4}-x_{2},$ $h_{j}^{0}(\bx)=0,$ $j=1,2,$ $h_{0}^{1}(\bx)=x_{1}^{2}+x_{2}^{2}-1,$ $h_{j}^{1}(\bx)=2x_{j},$ $j=1,2,$ and $\Omega_{1}$ is given by
\begin{align*}
		\Omega_{1} &=\left\{\left(y_{1}^{1}, y_{2}^{1}\right) \colon
		A= \left(\begin{array}{lll}
			1 & 0 & 0 \\
			0 & 1 & 0 \\
			0 & 0 & 1
		\end{array}\right)+y_{1}^{1}\left(\begin{array}{lll}
			0 & 0 & 1 \\
			0 & 0 & 0 \\
			1 & 0 & 0
		\end{array}\right)+y_{2}^{1}\left(\begin{array}{lll}
			0 & 0 & 0 \\
			0 & 0 & 1 \\
			0 & 1 & 0
		\end{array}\right)%=\begin{pmatrix}
%			1 & 0  & y_{1}^{1} \\
%			0 & 1  & y_{2}^{1}\\
%			y_{1}^{1} & y_{2}^{1} & 1
%		\end{pmatrix}
\in S_{ sdd}^{3}\right\} \\
		&=\left\{\left(y_{1}^{1}, y_{2}^{1}\right) \colon \exists \,  D= {\rm diag}\left ( d_1,d_2,d_3 \right ) \textrm{ s.t. } DAD \in S_{dd}^{3}, d_{i}>0,i=1,2,3\right\} \\
		&=\left\{\left(y_{1}^{1}, y_{2}^{1}\right) \colon d_{1} \geq d_3 \left |y_1^1  \right | , d_{2} \geq d_3 \left |y_2^1  \right |,  d_{3} \geq d_1 \left |y_1^1  \right |+d_2 \left |y_2^1  \right | \right\} \\
		&=\left\{\left(y_{1}^{1}, y_{2}^{1}\right) \colon (y_{1}^{1})^{2}+(y_{2}^{1})^{2}\leq1  \right\}.
\end{align*}
%where $I_{k}$ denotes a identity matrix of order $k.$
We see that for each $(y_{1}^{1},y_{2}^{1})\in\Omega_{1},$ $h_{0}^{1}(\bx)+y_{1}^{1}h_{1}^{1}(\bx)+y_{2}^{1}h_{2}^{1}(\bx)$ is a first-order SDSOS-convex polynomial.
Let $f_{0}(\bx)= x_{1}^{4}-x_{2}$ and $f_{1}(\bx)=x_{1}^{2}+x_{2}^{2}+2\left\|\left(x_{1}, x_{2}\right)\right\|_{2}-1.$
Then, $f_{1}(\bx)=\sup_{\left(y_{1}^{1}, y_{2}^{1}\right) \in \Omega_{1}}\left\{h_{0}^{1}(\bx)+\right.   \left.y_{1}^{1} h_{1}^{1}(\bx)+y_{2}^{1} h_{2}^{1}(\bx)\right\},$ and so, $f_{1}$ is a first-order SDSOS-convex semi-algebraic function.
Moreover, $f_{0}(\bx)$ is a first-order SDSOS-convex polynomial and so it is also a first-order SDSOS-convex semi-algebraic function.

Let $\widehat{\bx}=(1/3,0).$
Then, it can be verified that $f_{1}(\widehat{\bx})=-2/9 < 0,$ which means that the Slater condition is satisfied, and so, assumption~{\bf (A1)} is also satisfied.
Moreover, letting $(\widehat{y}_{1}^{1},\widehat{y}_{2}^{1})=(0,0),$ so assumption~{\bf (A2)} holds.
	
We now consider the following relaxation dual problem for the problem~\eqref{P1}:
\begin{align*}
\sup\limits_{\substack{\gamma,\lambda_{0}^{1}, \lambda_{j}^{i}}} \ \ & \gamma \\
\textrm{s.t.}\quad & \lambda_{0}^{0}h^{0}_{0}+\lambda_{0}^{1} h_{0}^{1}+\sum_{j=1}^{2} \lambda_{j}^{1} h_{j}^{1}-\gamma \in \tilde \Sigma_{4}^{2}[\bx],\\
& \left(\begin{array}{ccc} \lambda_{0}^{1} & 0 & \lambda_{1}^{1} \\  0 & \lambda_{0}^{1} & \lambda_{2}^{1} \\ \lambda_{1}^{1} & \lambda_{2}^{1} & \lambda_{0}^{1} \\ \end{array}\right) \in S_{sdd}^{3},\,\\
&\lambda_{0}^{0}=1, \ \lambda_{0}^{1}\geq0,\  \lambda_{j}^{1}\in\mathbb{R}, \  \  j = 1, 2.
\end{align*}
Invoking Proposition~\ref{proposition1}, there exists $Q\in S_{sdd}^{s(4)}(=S_{ sdd}^{15})$ such that
\begin{align}\label{ex1rel1}
\lambda_{0}^{0}h^{0}_{0}+\lambda_{0}^{1} h_{0}^{1}+\sum_{j=1}^{2} \lambda_{j}^{1} h_{j}^{1}-\gamma=\ \langle \lceil \bx\rceil_4\lceil \bx\rceil_4^{T},Q\rangle, \  \forall \bx\in\mathbb{R}^{2}.
\end{align}
Thanks to \cite[Theorem 1]{Reznick1978}, we can reduce the size of $\lceil \bx\rceil_4,$ that is, 6, and so $Q \in S^6_{+}.$
In more detail, $\lceil \bx\rceil_4=(1,x_{1},x_{2},x_{1}^{2},x_{1}x_{2},x_{2}^{2})^{T}$ in $(\ref{ex1rel1}).$
With this fact, we formulate the following dual problem for the problem~\eqref{P1},
\begin{align*}\label{SOCP1}
\sup\limits_{\substack{\gamma,Q}} \ \ & \gamma \tag{$\widehat{\rm Q}_{1}$}\\
\textrm{s.t.}\ \ \,   & -\lambda_{0}^{1}-\gamma=Q_{11}, \, \lambda^{1}_{1}=Q_{12}, \, -1+2\lambda^{1}_{2}=2Q_{13}, \\
		& \lambda^{1}_{0}=2Q_{14}+Q_{22},  \,\lambda^{1}_{0}=Q_{33}+2Q_{16}, \\
		& 1=Q_{44}, \, 0=Q_{15}=Q_{25}=Q_{35}=Q_{45}=Q_{55}, \\
		&  0=Q_{23}=Q_{24}=Q_{34}, \, 0=Q_{26}=Q_{36}=Q_{46}=Q_{56}=Q_{66}, \\
		& \left(\begin{array}{ccc} \lambda_{0}^{1} & 0 & \lambda_{1}^{1} \\  0 & \lambda_{0}^{1} & \lambda_{2}^{1} \\ \lambda_{1}^{1} & \lambda_{2}^{1} & \lambda_{0}^{1} \\ \end{array}\right)\in S_{sdd}^{3},  \\
		& \gamma\in\mathbb{R}, \,Q\in S_{sdd}^6, \,\lambda_{0}^{1}\geq0,\  \,\lambda_{j}^{1}\in\mathbb{R}, \, j=1, 2.
\end{align*}
Now, solving problem \eqref{SOCP1} using CVX \cite{Grant2013} in MATLAB gives us its optimal value $\gamma=-0.414214 \approx 1-\sqrt{2}.$
It follows from Theorem~\ref{thm1} that $\overline{\gamma}:=\gamma=-0.414214 \approx 1-\sqrt{2}$ is the optimal value of the problem \eqref{P1}.

To proceed, we formulate the following dual problem of the problem~\eqref{SOCP1},
\begin{align}\label{SOCP}
\inf\limits_{\substack{w\in\mathbb{R}^{15}\\Z_{i}\succeq0}} \ \ & \sum_{\ba\in\mathbb{N}^{n}_{2d}} w_{40}-w_{02}\tag{${\rm Q}_{1}$}\\
{\rm s.t.}\ \ \
		&w_{20}+w_{02}-w_{00}+\langle A_{0}^{1},Z_{1}\rangle\leq0, \notag \\
		&2 w_{10}+\langle A_{1}^{1},Z_{1}\rangle=0, \notag\\
		&2 w_{01}+\langle A_{2}^{1},Z_{1}\rangle=0, \notag \\
		%		& \textcolor[rgb]{1.00,0.00,0.00}{\left(\frac{\left(Z_{1}\right)_{r r}-\left(Z_{1}\right)_{s s}}{2}\right)^{2}+\left(\left(Z_{1}\right)_{r s}\right)^{2} \leq\left(v_{1}^{r s}\right)^{2}, 1 \leq r<s \leq 3,}\notag\\
		%		& \textcolor[rgb]{1.00,0.00,0.00}{\left(Z_{1}\right)_{r r}+\left(Z_{1}\right)_{s s} \geq 2 v_{1}^{r s}, 1 \leq r<s \leq 3,}\notag\\
		%&\left \| \binom{\frac{(Z_{1})_{rr}-(Z_{1})_{ss}}{2} }{(Z_{1})_{rs}}  \right \| \le v_{1}^{rs},(Z_{1})_{rr}+(Z_{1})_{ss}\le 2v_{1}^{rs},1\le r\le s\le 3, \notag\\
		& \left\|\left(
		\begin{array}{c}
			2(Z_{1})_{rs} \\
			(Z_{1})_{rr} - (Z_{1})_{ss} \\
		\end{array}
		\right)\right\|\leq (Z_{1})_{rr}+(Z_{1})_{ss}, \ 1\leq r<s\leq 3,  \nonumber \\
		&\mathbf{M}_{l}(\bw)=\left(\begin{array}{llllll}
			w_{00} & w_{10} & w_{01} & w_{20} & w_{11} & w_{02} \\
			w_{10} & w_{20} & w_{11} & w_{30} & w_{21} & w_{12} \\
			w_{01} & w_{11} & w_{20} & w_{21} & w_{12} & w_{03} \\
			w_{20} & w_{30} & w_{21} & w_{40} & w_{31} & w_{22} \\
			w_{11} & w_{21} & w_{12} & w_{31} & w_{22} & w_{13} \\
			w_{02} & w_{12} & w_{03} & w_{22} & w_{13} & w_{04}
		\end{array}\right) , \notag \\
		& \left\|\binom{2\left(\mathbf{M}_{l}(\bw)\right)_{i j}}{\left(\mathbf{M}_{l}(\bw)\right)_{i i}-\left(\mathbf{M}_{l}(\bw)\right)_{j j}}\right\| \leq\left(\mathbf{M}_{l}(\bw)\right)_{i i}+\left(\mathbf{M}_{l}(\bw)\right)_{j j}, 1 \leq i, j \leq 6 , \notag \\
		&  w_{\bze}=1.\nonumber
\end{align}
Solving the problem \eqref{SOCP} using CVX \cite{Grant2013} in MATLAB, we obtain the optimal value $1-\sqrt{2}$ and an optimal solution $(\overline \bw,\overline Z_{1})$ for the problem~\eqref{SOCP}, where
\begin{align*}
&\overline{\bw}=(1,0,0.4142,0,0,0,0,0,0,0,0,0,0,0,0)^{T},\\
&\overline{Z}_{1}\approx\left(
		\begin{array}{rrr}
			0 & 0 & 0 \\
			0 & 0.4142 & -0.4142 \\
			0 & -0.4142 & 0.4142 \\
		\end{array}
		\right).
\end{align*}
Thus, Theorem~\ref{thm2} implies that $\overline{\bx}=(0,0.4142)$ is an optimal solution to the problem~\eqref{P1}.
}\end{example}

%\section{Applications}\label{sect:5}

\section{Applications to Robust Optimization}\label{sect:5}
In this section, we study a class of robust optimization problems by using the results obtained in Section~\ref{sect:4}.
Among others, we focus on the following robust first-order SDSOS-convex optimization problem under constraint data uncertainty,
\begin{align}\label{robust}
\min_{\bx \in \mathbb{R}^{n}} & \quad f_{0}(\bx) \tag{RP} \\
\text{s.t.} & \quad g_{i}^{(0)}(\bx)+\sum_{j=1}^{t_{i}} u_{i}^{(j)} g_{i}^{(j)}(\bx)+\sum_{j=t_{i}+1}^{s_i} u_{i}^{(j)} g_{i}^{(j)}(\bx) \leq 0, \forall u_{i} \in \mathcal{U}_{i}, i=1, \ldots, m, \nonumber
\end{align}
where $f_0,$ $g_{i}^{(j)},\ i = 1, \ldots, m,\ j = 0, 1, \ldots, t_{i},$ are first-order SDSOS-convex polynomials with degree $d,$ $g_{i}^{(j)},\ i = 1, \ldots, m,\ j = t_{i}+1, \ldots, s_i,$ are affine functions, and $u_{i}$ are uncertain parameters belonging to uncertainty sets $\mathcal{U}_{i},$ $i = 1, \ldots, m.$
	
At this point, we would like to mention that Chieu et al. \cite{Chieu2018} has studied robust SOS-convex polynomial optimization problems under structured constraint uncertainty, namely, they first introduced the notion of restricted spectrahedron data uncertainty set, which is a convex compact set given by
\begin{align*}
\mathcal{U}_{i}^{r e s}=\left\{\left(u_{i}^{(1)}, \ldots, u_{i}^{\left(t_{i}\right)}, u_{i}^{\left(t_{i}+1\right)}, \ldots, u_{i}^{(s_i)}\right)\right. & \in \mathbb{R}^{s_i} \colon A_{i}^{0}+\sum_{j=1}^{s_i} u_{i}^{(j)} A_{i}^{j} \succeq 0, \\
		\left(u_{i}^{(1)}, \ldots, u_{i}^{\left(t_{i}\right)}\right) & \left.\in \mathbb{R}_{+}^{t_{i}},\ \left(u_{i}^{\left(t_{i}+1\right)}, \ldots, u_{i}^{(s)}\right) \in \mathbb{R}^{s_i-t_{i}}\right\} .
\end{align*}
Then, under Slater-type conditions, they showed that the robust SOS-convex optimization problem with such uncertainty admits an exact SDP relaxation and further proposed how to extract optimal solutions from this relaxation.

Here, we present a parallel extension of robust optimization where the underlying data satisfies first-order SDSOS-convexity.
Now, we introduce the notion of SDD-restricted spectrahedron uncertainty set which is a convex compact set given by
\begin{align*}
\mathcal{U}_{i}^{sdd}=\left\{\left(u_{i}^{(1)}, \ldots, u_{i}^{\left(t_{i}\right)}, u_{i}^{\left(t_{i}+1\right)}, \ldots, u_{i}^{(s_i)}\right)\right. & \in \mathbb{R}^{s_i} \colon A_{i}^{0}+\sum_{j=1}^{s_i} u_{i}^{(j)} A_{i}^{j} \in S_{sdd}^{q_{i}}, \\
	\left(u_{i}^{(1)}, \ldots, u_{i}^{\left(t_{i}\right)}\right) & \left.\in \mathbb{R}_{+}^{t_{i}},\ \left(u_{i}^{\left(t_{i}+1\right)}, \ldots, u_{i}^{(s_i)}\right) \in \mathbb{R}^{s_i-t_{i}}\right\},
\end{align*}
then we turn to the robust first-order SDSOS-convex optimization problem with SDD-restricted spectrahedron uncertainty set.
In other word, we restrict ourselves to the problem~\eqref{robust} with $\mathcal{U}_{i} = \mathcal{U}_{i}^{sdd}.$
Observe that this type of robust optimization problem can be seen as a first-order SDSOS-convex semialgebraic program.
To see this, we define $g_{i}\left(\bx, u_{i}\right) := g_{i}^{(0)}(\bx)+\sum\limits_{j=1}^{t_{i}} u_{i}^{(j)} g_{i}^{(j)}(\bx) + \sum\limits_{j=t_{i}+1}^{s_i} u_{i}^{(j)} g_{i}^{(j)}(\bx) ,$ and $f_{i}(\bx) := \sup\limits_{u_{i} \in \mathcal{U}_{i}^{sdd}}\left\{g_{i}\left(\bx, u_{i} \right)\right\}$ for $i = 1, \ldots, m.$
By the construction of the SDD-restricted spectrahedron uncertainty set, each $g_{i}\left(\cdot, u_{i}\right)$ is a first-order SDSOS-convex polynomial whenever $u_{i} \in \mathcal{U}_{i}^{sdd}.$ Furthermore, each SDD-restricted spectrahedron set $\mathcal{U}_{i}^{sdd}$ can be expressed as
\begin{equation}\label{costraint}
\mathcal{U}_{i}^{sdd}=\left\{\left(u_{i}^{(1)}, \ldots, u_{i}^{\left(t_{i}\right)}, u_{i}^{\left(t_{i}+1\right)}, \ldots, u_{i}^{(s_i)}\right) \in \mathbb{R}^{s} \colon \widetilde{A}_{i}^{0}+\sum_{j=1}^{s_i} u_{i}^{(j)} \widetilde{A}_{i}^{j} \in S_{sdd}^{q_{i}} \right\},
\end{equation}
where
$\tilde{A}_{i}^{0} =\left(\begin{array}{cc}
			0_{t_{i} \times t_{i}} & 0 \\
			0 & A_{i}^{0}
		\end{array}\right),$
$\tilde{A}_{i}^{j}=\left(\begin{array}{cc}
			\operatorname{diag} e^{j} & 0 \\
			0 & A_{i}^{j}
		\end{array}\right)$
for $j =1, \ldots, t_{i},$
$\tilde{A}_{i}^{j}  =
		\left(\begin{array}{cc}
			0_{t_{i} \times t_{i}} & 0 \\
			0 & A_{i}^{j}
		\end{array}\right)$ for $j = t_{i}+1, \ldots, s_i,$
and $e^{j} \in \mathbb{R}^{t_{i}}$ denotes the vector whose $j$-th element equals to one and $0$ otherwise.
%Consequently, the robust optimization problem (RP) with $ \mathcal{U}_{i}=\mathcal{U}_{i}^{sdd}$ can be cast as an first-order SDSOS-convex semialgebraic program.
This reformulation allows us to write the robust problem with a corresponding relaxation problem,
\begin{align}\label{robust-r}
\sup_{\gamma, \lambda_{0}^{i}, \lambda_{j}^{i}} & \ \gamma \tag{${\rm \hat{Q}}$} \\
\text { s.t. } & \ f + \sum_{i=1}^{m}\left(\lambda_{0}^{i} g_{0}^{i}+\sum_{j=1}^{s_{i}} \lambda_{j}^{i} g_{j}^{i}\right)-\gamma \in \widetilde{\Sigma}[\mathbf{x}]_{2 d}, \nonumber \\
		& \ \lambda_{0}^{i} \tilde{A_{0}^{i}}+\sum_{j=1}^{s_{i}} \lambda_{j}^{i} \tilde{A_{j}^{i}} \in S_{s d d}^{q_{i}} , \nonumber \\
		& \ \lambda_{0}^{i} \geq 0, \ \lambda_{j}^{i} \in \mathbb{R},\ j = 1, \ldots, s_{i},\  i =  1, \ldots, m. \nonumber
\end{align}
The dual problem of the problem~\eqref{robust-r} is formulated as follows,
\begin{align}\label{robust-dual}
\inf_{\substack{\mathbf{w} \in \mathbb{R}^{s(n, 2 d)} \\  Z_{i} \in S^{q i}, i \in I}} & \ \sum_{\alpha \in \mathbb{N}_{2 d}^{n}} f_{\alpha} w_{\alpha} \tag{Q} \\
\text { s.t. } &\ \sum_{\alpha \in \mathbb{N}_{2 d}^{n}}\left(g_{0}^{i}\right)_{\alpha} w_{\alpha}+\left\langle \tilde{A_{0}^{i}}, Z_{i}\right\rangle \leq 0, \ i=1, \ldots, m, \nonumber \\
	 	&\ \sum_{\alpha \in \mathbb{N}_{2 d}^{n}}\left(g_{j}^{i}\right)_{\alpha} w_{\alpha}+\left\langle \tilde{A_{j}^{i}} , Z_{i}\right\rangle=0,\ i=1, \ldots,m,\ j=1, \ldots, s_{i}, \nonumber \\
	 	&\ \left\|\binom{2\left(Z_{i}\right)_{r s}}{\left(Z_{i}\right)_{r r}-\left(Z_{i}\right)_{s s}}\right\| \leq\left(Z_{i}\right)_{r r}+\left(Z_{i}\right)_{s s},\ 1 \leq r<s \leq q_{i},\ i=1, \ldots,m , \nonumber \\
	 	& \ \left\|\binom{2\left(\mathbf{M}_{d}(\mathbf{w})\right)_{i j}}{\left(\mathbf{M}_{d}(\mathbf{w})\right)_{i i}-\left(\mathbf{M}_{d}(\mathbf{w})\right)_{j j}}\right\| \leq\left(\mathbf{M}_{d}(\mathbf{w})\right)_{i i}+\left(\mathbf{M}_{d}(\mathbf{w})\right)_{j j},\ 1 \leq i<j \leq s(n, d), \nonumber \\
	 	& \ w_{0}=1 .\nonumber
\end{align}
	
Up to now, during the above process, we have established an exact SOCP relaxation for the robust first-order SDSOS-convex optimization problem with SDD-restricted spectrahedron data uncertainty, and demonstrated how to recover an optimal solution from its SOCP relaxation.
As a direct consequence of Theorems~\ref{thm1} and \ref{thm2}, we have the following result.
\begin{corollary}\label{final_cor}
Consider the problem~\eqref{robust} under SDD-restricted spectrahedron data uncertainty$,$ i.e.$,$ $\mathcal{U}_{i}=\mathcal{U}_{i}^{sdd}$ where $\mathcal{U}_{i}^{sdd}$ are convex compact sets given as in \eqref{costraint}.
Suppose that the convex cone
\begin{equation}\label{convex-cone}
\bigcup_{\mathbf{y}^{i} \in \Omega_{i}, \lambda_{i} \geq 0} \operatorname{epi}\left(\sum_{i=0}^{m} \lambda_{i}\left(g_{i}^{(0)}+\sum_{j=1}^{s_{i}} u_{i}^{(j)} g_{i}^{(j)}\right)\right)^{*}
\end{equation}
is closed.
Then$,$ the following statements hold.
\begin{enumerate}[\upshape (i)]
\item $\inf\eqref{robust} = \sup\eqref{robust-r}.$
\item Suppose in addition that$,$ for each $i = 1, \ldots, m,$ there exist $\hat{\mathbf{y}}^{i} \in \mathbb{R}^{s_{i}}$ and a diagonal matrix $D_{i}$ with all positive diagonal entries such that $D^{i}\left(\tilde{A_{0}^{i}}+\sum_{j=1}^{s_{i}} \hat{y}_{j}^{i} \tilde{A_{j}^{i}}\right) D^{i}$ is strictly diagonally dominant.
Then$,$ we have
\begin{align*}
\inf\eqref{robust} = \sup\eqref{robust-r} = \inf\eqref{robust-dual}.
\end{align*}
Let $(\overline{\mathbf{w}}, \overline{Z}_{0},\ldots,\overline{Z}_{m})$ be an optimal solution to the problem~\eqref{robust-dual}.
Then
\begin{align*}
\overline{\bx} = \left(L_{\overline{\mathbf{w}}}(x_{1}), \ldots, L_{\overline{\mathbf{w}}}(x_{n})\right)
\end{align*}
is an optimal solution to the problem \eqref{robust}.
\end{enumerate}
\end{corollary}

\begin{remark}{\rm
Note that the cone \eqref{convex-cone} in Corollary~\ref{final_cor} is convex due to the concavity of $g_i$ with respect to the parameters $u_i$; see, e.g., \cite[Proposition 2.3]{Jeyakumar2010}.
}\end{remark}

\section{Conclusions}\label{sect:6}
In this paper, we have introduced a new class of {\it nonsmooth} convex functions --- first-order SDSOS-convex semi-algebraic functions.
Under suitable assumptions, we have proved that the optimal value and optimal solutions of optimization problems with first-order SDSOS-convex semi-algebraic functions can be obtained by solving an SOCP problem.
The results rely on the representation theory of nonnegative polynomials under SDSOS polynomials and utilize the associated Jensen's inequality to recover the global optimal solution from the SOCP relaxation.

\subsection*{Acknowledgements}
This work was supported by the National Research Foundation of Korea (NRF) grant funded by the Korea government (MSIT) (NRF-2021R1C1C2004488).
Liguo Jiao and Jae Hyoung Lee contributed equally as corresponding authors.

\subsection*{Availability of data and materials} The codes and datasets generated in this study are available from the first author on reasonable request.

\subsection*{Disclosure statement}
The authors have no conflict of interest to declare that are relevant to the content of this article.

%\section*{Declarations}

%\def\refname{\Large\bfseries References}
%\bibliographystyle{abbrv}
%\bibliography{sdsos-sa}

\end{document}